\documentclass{article}
\usepackage{amssymb,latexsym,amsmath,amsthm}
\usepackage{fancyhdr}

\usepackage{color}
\usepackage{mathrsfs}

\usepackage{lipsum}

\newcommand\blfootnote[1]{%
  \begingroup
  \renewcommand\thefootnote{}\footnote{#1}%
  \addtocounter{footnote}{-1}%
  \endgroup
}

\numberwithin{equation}{section}
\theoremstyle{plain}
\newtheorem{lemma}{Lemma}
\newtheorem{theorem}[lemma]{Theorem}

\theoremstyle{definition}
\newtheorem*{remark*}{Remark}

\numberwithin{lemma}{section}

\newcommand{\ep}{\varepsilon}
\newcommand{\N}{\mathbb{N}}
\newcommand{\R}{\mathbb{R}}
\newcommand{\Z}{\mathbb{Z}}
\newcommand{\Q}{\mathbb{Q}}

\newcommand{\CC}{\mathcal{C}}
\newcommand{\twosum}[2]{\sum_{\substack{#1\\#2}}}
\newcommand{\beql}[1]{\begin{equation}\label{#1}}
\newcommand{\eeq}{\end{equation}}

\newcommand{\modd}[1]{\; ( \text{mod} \; #1)}

\newcommand{\sL}{\mathsf{\Lambda}}

\newcommand{\card}{\#}

\begin{document}

\title{Counting Square-full Solutions to $x+y=z$}  
\author{D. R. Heath-Brown\\
Mathematical Institute, Oxford }
            
  \date{}
  \maketitle
  
\begin{abstract}
We show that there are $O(B^{3/5-3/1555+\ep})$ triples $(x,y,z)$ of square-full integesr up to $B$
satisfying the equation $x+y=z$ for any fixed $\ep>0$. This is the first improvement over the `easy'
exponent $3/5$, given by Browning and Van Valckenborgh \cite{BVV}. One new tool is a strong uniform bound
for the counting function for equations $aX^3+bY^3=cZ^3$.
\end{abstract}
\noindent
\newpage

\section{Introduction}

This paper is concerned with positive integer solutions to the equation $u+v=w$ in which $u,v$ and $w$ 
are square-full. We remind the reader that a positive integer $n$ is said to be square-full if $p^2|n$ for 
every prime $p$ dividing $n$. One may easily check that such an integer may be written uniquely as 
$x^2y^3$ where $x$ and $y$ are positive integers, with $y$ being square-free.

Apart from the intrinsic interest of square-full solutions to $u+v=w$, the equation arises in the study of
rational points on Campana orbifolds. For the divisor
\[\Delta=\tfrac12[0]+\tfrac12[1]+\tfrac12[\infty]\]
the density of rational points of height at most $B$ on the orbifold $(\mathbb{P}^1,\Delta)$ can be understood
via the counting function
\[n_{\mathrm{prim}}(B)=\card\{(u,v,w)\in\N_{\mathrm{prim}}^3: u+v=w;\, u,v,w\mbox{ square-full};\, w\le B\},\]
where $\N_{\mathrm{prim}}^3$ is the set of triples $(u,v,w)\in\N^3$ with $\gcd(u,v,w)=1$. The question
of the behaviour of $n_{\mathrm{prim}}(B)$ seems first to have been raised by Poonen \cite{poonen}, who asks
whether $n_{\mathrm{prim}}(B)\sim B^{1/2}$, but the reader should note that Poonen is not using the notation
$\sim$ in precisely the same sense as is usual in analytic number theory. 

The function $n_{\mathrm{prim}}(B)$ was investigated further by Browning and Van Valckenborgh \cite{BVV},
who conjectured that $n_{\mathrm{prim}}(B)\sim c_0B^{1/2}$ for a certain explicit constant $c_0>0$. They 
viewed the equation $u+v=w$ as $x_1^2y_1^3+x_2^2y_2^3=x_3^2y_3^3$, giving
\[n_{\mathrm{prim}}(B)=\sum_{y_1,y_2,y_3}\card\{(x_1,x_2,x_3): x_1^2y_1^3+x_2^2y_2^3=x_3^2y_3^3\},\]
and noted that, for fixed values of $y_1,y_2,y_3$ there are asymptotically $c(y_1,y_2,y_3)B^{1/2}$ solutions
$(x_1,x_2,x_3)$, for appropriate constants $c(y_1,y_2,y_3)$. This produces their conjecture, with
\[c_0=\sum_{y_1,y_2,y_3}c(y_1,y_2,y_3).\]
They showed further that
\[\liminf_{B\to\infty} B^{-1/2}n_{\mathrm{prim}}(B)\ge c_0,\]
so that it is the question of upper bounds for $n_{\mathrm{prim}}(B)$ which remains. Here Browning and Van 
Valckenborgh showed that
\[n_{\mathrm{prim}}(B)\ll_\ep B^{3/5+\ep}\]
for any fixed $\ep>0$. They discuss the possibility of improvements, and comment that obtaining an upper bound
of the expected order of magnitude is ``much more challenging''. The question was investigated further by 
Zhao \cite{zhao}, who refined the estimate of Browning and Van Valckenborgh, showing that
\[n_{\mathrm{prim}}(B)\ll B^{3/5}(\log B)^9.\]

The goal of the present paper is to reduce the exponent $3/5$. It turns out that little extra effort is required to 
include solutions in which $\gcd(u,v,w)>1$. Thus we have the following result. 

\begin{theorem}\label{thm:main}
Let $n(B)$ be the number of triples $u,v,w\in\N$ of square-full integers, such that $u+v=w$ and
$w\le B$. Then
\[n(B)\ll_\ep B^{3/5-3/1555+\ep}\]
for any fixed $\ep>0$.
\end{theorem} 

The improvement, by $3/1555$, is disappointingly small. However it is hoped that the ideas in this
paper may inspire others to make more substantial progress now that an initial dent in the problem 
has been made.

Theorem \ref{thm:main} will be an easy consequence of the following result.
\begin{theorem}\label{thm2}
Let $a_1,a_2,a_3$ be non-zero square-free integers (of either sign) which are coprime in
pairs. Let $n(B;a_1,a_2,a_3)$ be the number of positive integer solutions
$x_1,x_2,x_3,y_1,y_2,y_3$ of the equation
\beql{axy}
a_1x_1^2y_1^3+a_2x_2^2y_2^3+a_3x_3^2y_3^3=0
\eeq
with $a_1y_1,a_2y_2$ and $a_3y_3$ square-free, and for which
\beql{range}
|a_1|x_1^2y_1^3,|a_2|x_2^2y_2^3,|a_3|x_3^2y_3^3\le B
\eeq
and
\beql{gcd}
\gcd(x_1y_1\,,\,x_2y_2\,,\,x_3y_3)=\gcd(a_1a_2a_3\,,\,x_1x_2x_3y_1y_2y_3)=1.
\eeq
Then $n(B;a_1,a_2,a_3)\ll_\ep B^{3/5-3/1555+\ep}$, for any fixed $\ep>0$.
\end{theorem}
To see how Theorem \ref{thm:main} follows from
Theorem \ref{thm2} suppose that $u+v=w$ with positive square-full integers $u,v,w$,
and let $h=\gcd(u,v,w)$. Then $h$ must be square-full, and $uh^{-1}$ must factor as $\prod p^e$
with $e\ge 2$ except possibly for primes $p| h$. Since every square-full integer may be 
written as a product $x^2y^3$ with $y$ square-free, we conclude that $u=ha_1x_1^2y_1^3$ for some 
integer $a_1|h$ coprime to $x_1y_1$, with $a_1y_1$ square-free. We may argue similarly for $v$ 
and $w$ and conclude that
\begin{eqnarray*}
n(B)&\le&\twosum{h\le B}{h\;\mathrm{square-full}}\sum_{a_1,a_2,a_3|h}n(B/h;a_1,a_2,a_3)\\
&\ll_\ep&\twosum{h\le B}{h\;\mathrm{square-full}} \left(\frac{B}{h}\right)^{3/5-3/1555+\ep}\tau(h)^3\\
&\ll_\ep& B^{3/5-3/1555+\ep},
\end{eqnarray*}
as required for Theorem \ref{thm:main}. Here $a_1,a_2,a_3$ are allowed to take both positive 
and negative values, and we use the fact that if $h$ runs over square-full
integers then $\sum_h h^{-\theta}<\infty$ whenever $\theta>\tfrac12$.
\bigskip

The main element in our strategy to prove Theorem \ref{thm2} is to consider the bi-homgeneous
congruence
\[a_1x_1^2y_1^3+a_2x_2^2y_2^3\equiv 0\modd{y_3^3}.\]
After various manoeuvres we find that $(x_1,x_2,y_1,y_2)$ lies on one of a small number of lines
in $\mathbb{P}^3(\mathbb{Q})$ or, more precisely, on 2-dimensional lattices $\mathsf{M}\subset\Z^4$.
These produce equations of the shape 
\[F(u,v)=-a_3y_3^3\square,\]
where $F$ is a form of degree 5, and these are handled via the ``square-sieve''.  One needs to know 
about the size of $\det(\mathsf{M})$ for these lattices, which determines the range for the variables 
$u$ and $v$ above. Throughout these steps there are various special cases that arise which require
separate treatment.

One auxiliary result deserves mention here.
\begin{theorem}\label{cubic}
Let $\rho(B;a,b,c)$ be the number of integer solutions $x,y,z$ to the equation
\beql{ceq}
ax^3+by^3+cz^3=0
\eeq	
with $\gcd(x,y,z)=1$ and $\max(|x|,|y|,|z|)\le B$. Then
\[\rho(B;a,b,c)\ll_\ep B^\ep\]
for any $\ep>0$, uniformly for all non-zero integer $a,b,c$.
\end{theorem}
It is clear from the proof that one may replace $B^\ep$ by $B^{\ep(B)}$ with some explicit function
$\ep(B)$ tending to zero, but we shall not use this fact.  

For a general cubic form $C(x,y,z)$ with at least one non-trivial integral zero the corresponding 
counting function $\rho_C(B)$ will be asymptotically $c_B(\log B)^{r/2}$, where $r$ is the rank
of the Jacobian of $C$. It then follows that $\rho_C(B)\ll_{C,\ep}B^\ep$. It was shown by the author
\cite[Theorem 3]{HBcub} that one can remove the dependence on $C$ under a certain
``Rank Hypothesis''. This states that if $r(E)$ and $C_E$ denote the rank and conductor of an
elliptic curve $E$, then $r(E)=o(\log C_E)$. What is
special about our particular situation is that we can estimate the rank of the relevant elliptic curves
quite efficiently.

{\em Acknowledgement.} The author is particularly grateful to Tim Browning, for introducing 
him to this problem, explaining its background, providing references, and for pointing out
minor errors in an earlier draft of this paper.

\section{Auxiliary Results}\label{AR}

We begin  by proving Theorem \ref{cubic}.  As mentioned in the introduction, we will use a good bound
for the ranks of certain elliptic curves.
\begin{lemma}\label{rank}
Let $M\in\N$ be cube-free. Then if we write $r(M)$ for the rank of the elliptic curve
\[E_M: Y^2=X^3-3M^2 \]
we will have
\[r(M)\ll\ \omega(M)+1\ll\frac{\log 3M}{\log\log 3M}.\]
\end{lemma}

We first show how to deduce Theorem \ref{cubic}, and then go on to establish Lemma \ref{rank}.
\begin{proof}[Proof of Theorem \ref{cubic}]
Clearly we may suppose that $\gcd(a,b,c)=1$. The result is trivial unless there are at least two
non-proportional solutions $(x_1,y_1,z_1)$ and $(x_2,y_2,z_2)$, in which case $(a,b,c)$ is 
orthogonal to both $(x_1^3,y_1^3,z_1^3)$ and $(x_2^3,y_2^3,z_2^3)$. It then follows that $(a,b,c)$
is proportional to the vector product
\beql{wed}
(x_1^3,y_1^3,z_1^3)\wedge (x_2^3,y_2^3,z_2^3)=(y_1^3z_2^3-z_1^3y_2^3\,,
\,z_1^3x_2^3-x_1^3z_2^3\,,\,x_1^3y_2^3-y_1^3x_2^3).
\eeq
Since $\gcd(a,b,c)=1$ we deduce that (\ref{wed}) is an integer multiple of $(a,b,c)$, whence
\[\max(|a|,|b|,|c|)\le 2B^6.\]

We now follow the argument given by the author \cite[Proof of Theorem 3]{HBcub}. 
Suppose that $12abc=n^3M$ for some cube-free integer $M$. Then we may take the Jacobian of the 
curve (\ref{ceq}) to be $E_M$. If we write $C(M)$ for the conductor of $E_M$, the argument of 
\cite[Proof of Theorem 3]{HBcub} (see in particular the third displayed formula on 
\cite[page 24]{HBcub}) will yield
\[\rho(B;a,b,c)\ll \exp\left\{A\frac{r(M)\log B}{\log C(M)}\right\}\]
for some absolute constant $A$, where $r(M)$ is the rank of $E_M$ as in Lemma \ref{rank}. 
Since $M$ is cube-free we will have $\log C(M)\gg\log 3M$, whence Lemma \ref{rank} yields
\[\rho(B;a,b,c)\ll \exp\left\{O\left(\frac{\log B}{\log\log M}\right)\right\}.\]
This is sufficient for the theorem when $M\ge M(\ep)$ with a suitably large $M(\ep)$. For the remaining 
cases we use the fact that
\[\rho(B;a,b,c)\ll_{\ep,a,b,c} (\log B)^{r(M)/2}.\]
This follows from the argument leading up to \cite[Lemma 3]{HBcub}, together with N\'{e}ron's estimate
for the number of points of bounded height on an elliptic curve, see \cite[(12)]{HBcub} for example. Since
$a,b,c,M$ are all bounded in terms of $\ep$ this is enough to complete the proof of the theorem.
\end{proof}
\bigskip

We now move to the proof of Lemma \ref{rank}. The key point is that we can bound $r(M)$
efficiently using descent via 3-isogeny, as described by Cohen and Pazuki
\cite{CP} for example.  
\begin{proof}[Proof of Lemma \ref{rank}]
In the notation of Cohen and Pazuki, \cite[Proposition 2.2]{CP} states that
\beql{bec1}
3^{r(M)+1}=\card\mathrm{Im}(\alpha)\,\card\mathrm{Im}(\hat{\alpha}),
\eeq
where $\alpha$ and $\hat{\alpha}$ are the 3-descent maps for the curves $E_M$ and 
\[\widehat{E_M}:Y^2=X^3+81M^2\]
respectively. We can estimate $\card\mathrm{Im}(\hat{\alpha})$
by \cite[Theorem 3.1]{CP}, since $D$ (in the notation from \cite[Lemma 1.2]{CP}) will be 1. 
Moreover ``$b$'' is $9M$, and we see from \cite[Theorem 3.1, part (3)]{CP} that
\[\card\mathrm{Im}(\hat{\alpha})\le\card\{(u_1,u_2)\in\N^2: u_1u_2|18M\}=\tau_3(18M),\]
whence
\beql{bec2}
\log\left(\card\mathrm{Im}(\hat{\alpha})\right)\ll 1+\omega(M).
\eeq
We can estimate $\card\mathrm{Im}(\alpha)$ similarly using \cite[Theorem 4.1, part (1)]{CP},
where we will have $D=-3$.  Write $K=\Q(\sqrt{-3})$ and let $G_3$ be the 
subgroup of $K^*/{K^*}^3$ consisting of classes $[u]$ for which $N_{K/\Q}(u)\in{\Q^*}^3$. Then if 
$[u]\in G_3$ we may write $u=v^2\tau(v)$ for some $v\in K^*$, where $\tau$ is the non-trivial 
automorphism of $K$. With this notation \cite[Theorem 4.1, part (1)]{CP} says that 
$\mathrm{Im}(\alpha)$ consists of those classes $[u]=[v^2\tau(v)]$ such that the equation
\[2v_2X^3+2Dv_1Y^3+\frac{2M}{v_1^2-Dv_2^2}Z^3+6v_1X^2Y+6v_2DXY^2=0\]
has a non-zero solution $X,Y,Z\in\Z$, where $v=v_1+v_2\sqrt{-3}$. We now observe that 
any class $[u]\in G_3$ has a representation $[u]=[v^2\tau(v)]$ in which $v$ is a cube-free 
algebraic integer with no non-trivial divisor in $\N$. Moreover the equation above is equivalent to 
\[v^2\tau(v)W^3-v\tau(v)^2\tau(W)^3+2\sqrt{-3}MZ^3=0,\]
where $W=X+Y\sqrt{-3}$. If $\pi|N_{K/\Q}(v)$ is a prime of $K$ not dividing $6M$ then it divides 
exactly one of $v$ or $\tau(v)$, so that the above equation will not be locally solvable at $\pi$.
Thus $v$ must be composed solely of primes $\pi$ dividing $6M$, whence
\[\card\mathrm{Im}(\alpha)\le 3^{\omega_K(6M)}\le 3^{2+2\omega(M)}.\]
The lemma then follows from (\ref{bec1}) and (\ref{bec2}).
\end{proof}

Our next auxiliary result is an application of the ``square-sieve'' (see the author \cite{HBss}).
\begin{lemma}\label{quintic}
Let $F(X,Y)\in\Z[X,Y]$ be a non-zero form of odd degree $D\ge 3$, having no repeated factor. 
Let $U,V\ge 1$ and let $\mu_0(U,V)$ be the number of pairs of integers 
$x\in[-U,U]$ and  $y\in[-V,V]$ for which $F(x,y)$ is a non-zero square. Then 
\[\mu_0(U,V)\ll_D (UV)^{2/3}\log^2(2||F||)+U+V,\]
where $||F||$ is the maximum modulus of the coefficients of $F$.
\end{lemma}
\begin{proof}
Let $\Delta_F\not=0$ be the discriminant of $F$, and let $P\ge 2$
be a parameter to be decided later. Then if $F(x,y)$ is a non-zero square, and $p$ runs over primes, we have
\[\twosum{P<p\le 2P}{p\nmid \Delta_F}\left(\frac{F(x,y)}{p}\right)\ge \pi(2P)-\pi(P)-\omega(\Delta_F F(x,y)),\]
so that 
\[\twosum{P<p\le 2P}{p\nmid \Delta_F}\left(\frac{F(x,y)}{p}\right)\gg \frac{P}{\log P}\]
provided that $P\gg_D \log(2||F||UV)$, as we henceforth assume.  It follows that
\[\left(\frac{P}{\log P}\right)^2 \mu_0(U,V)\ll \sum_{(x,y)\in\Z^2}\exp\{-x^2/U^2-y^2/V^2\}
\left\{\twosum{P<p\le 2P}{p\nmid\Delta_F}\left(\frac{F(x,y)}{p}\right)\right\}^2.\]
If we expand the square on the right we will have to examine
\[S_0(q):=\sum_{(x,y)\in\Z^2}\exp\{-x^2/U^2-y^2/V^2\}\left(\frac{F(x,y)}{q}\right),\]
with $q=p_1p_2$.  When $p_1=p_2$ this is $O(UV)$, and we conclude that
\beql{res}
\mu_0(U,V)\ll UV(\log P)P^{-1}+(\log P)^2P^{-2}\twosum{P<p_1\not=p_2\le 2P}{p_1,p_2\nmid\Delta_F}|S_0(p_1p_2)|.
\eeq
To estimate $S_0(q)$ we use Poisson summation to show that
\begin{eqnarray}\label{feed}
S_0(q)&=&\sum_{(m,n)\in\Z^2}\left(\frac{F(m,n)}{q}\right)
\sum_{(r,s)\in\Z^2}\exp\{-(m+qr)^2/U^2-(n+qs)^2/V^2\}\nonumber\\
&=&\frac{\pi UV}{q^2}\sum_{(u,v)\in\Z^2}\exp\{-\pi^2u^2U^2q^{-2}-\pi^2v^2V^2q^{-2}\}\Sigma_0(u,v;q),
\end{eqnarray}
say, with
\[\Sigma_0(u,v;q)=\sum_{m,n\modd{q}}\left(\frac{F(m,n)}{q}\right)e_q(mu+nv).\]
The usual multiplicativity argument shows that
\beql{PP}
\Sigma_0(u,v;p_1p_2)=\left(\frac{p_2}{p_1}\right)^D\left(\frac{p_1}{p_2}\right)^D\Sigma_0(u,v;p_1)\Sigma_0(u,v;p_2)
\eeq
when $p_1\not=p_2$, so that it is enough to investigate $\Sigma_0(u,v;q)$ when $q$ is prime,
equal to $p$, say.

It looks at first glance as if we will need the theory of two-dimensional exponential sums, but fortunately we can
simplify things by the following argument.  If $t$ is an integer coprime to $p$ then $tm$ and $tn$ run over
all integers modulo $p$ when $m$ and $n$ do. We therefore have
\begin{eqnarray}\label{oe}
\Sigma_0(u,v;p)&=&\frac{1}{p-1}\sum_{t=1}^{p-1}\;\sum_{m,n\modd{p}}\left(\frac{F(tm,tn)}{p}\right)e_p(t\{mu+nv\})\nonumber\\
&\hspace{-2cm}=&\hspace{-1cm}\frac{1}{p-1}\sum_{m,n\modd{p}}\left(\frac{F(m,n)}{p}\right)\sum_{t=1}^{p-1}\left(\frac{t}{p}\right)^D e_p(t\{mu+nv\}).
\end{eqnarray}
Since $D$ is odd we get
\[\Sigma_0(u,v;p)=\frac{\tau_p}{p-1}\sum_{m,n\modd{p}}\left(\frac{F(m,n)}{p}\right)\left(\frac{mu+nv}{p}\right),\]
where $\tau_p=\sqrt{p}$ or $i\sqrt{p}$ is the Gauss sum. (Note that this remains true even when $p|mu+nv$.)
The form $F(X,Y)(Xu+Yv)$ cannot be a constant multiple of a square modulo $p$, unless $u$ and $v$ vanish 
modulo $p$, since $p\nmid\Delta_F$, and $D\not=1$. Thus 
\[\sum_{m,n\modd{p}}\left(\frac{F(m,n)}{p}\right)\left(\frac{mu+nv}{p}\right)\ll_D p^{3/2}\]
for all $u,v$, by Weil's Riemann hypothesis for curves over $\mathbb{F}_p$. We therefore have $\Sigma_0(u,v;p)\ll_D p$,
whence the product formula (\ref{PP}) yields $\Sigma_0(u,v;q)\ll_D q$ for $q=p_1p_2$. We can now feed 
this bound into (\ref{feed}) to deduce that
\[S_0(q)\ll_D \frac{UV}{q^2}\left(\frac{q}{U}+1\right)\left(\frac{q}{V}+1\right)q=q^{-1}(q+U)(q+V).\]
Then, since $q=p_1p_2$ will be of order $P^2$, we see that (\ref{res}) yields
\[\mu_0(U,V)\ll_D UV(\log P)P^{-1}+P^2+U+V\]
whenever $P\gg_D\log(2||F||UV)$. We therefore choose $P$ to be a suitable constant multiple of 
$(UV)^{1/3}\log(2||F||)$, and the result follows.
\end{proof}

As a companion to Lemma \ref{quintic} we have the following result for univariate polynomials .
\begin{lemma}\label{quintic2}
Let $f(X)\in\Z[X]$ be a polynomial of degree $D$, and suppose that $f$ is not a constant multiple of a square. 
Let $U\ge 1$ and let $\mu_1(U)$ be the number of integers 
$x\in[-U,U]$ for which $f(x)$ is a square. Then 
\[\mu_1(U)\ll_D U^{1/2}.\]
\end{lemma}
This is an immediate consequence of Theorem 1.1 of Vaughan \cite{vaughan}.
\bigskip

From Lemmas \ref{quintic} and \ref{quintic2} we have the following corollary.
\begin{lemma}\label{q3}
Let $F(X,Y)\in\Z[X,Y]$ be a non-zero form of odd degree $D\ge 3$, and suppose that  
$F$ has no repeated factor. 
Let $U,V\ge 1$ and let $\mu(U,V)$ be the number of pairs of integers 
$x\in[-U,U]$ and  $y\in[-V,V]$ for which $F(x,y)$ is a non-zero square. Then 
\[\mu(U,V)\ll_D (UV)^{2/3}\log^2(2||F||),\]
where $||F||$ is the maximum modulus of the coefficients of $F$.
\end{lemma}
\begin{proof}
We would like to apply Lemma \ref{quintic2} to the polynomial $f(X)=F(X,y)$ for a fixed
integer $y$. If $Y\nmid F(X,Y)$ then $f$ will have odd degree $D$, whence Lemma \ref{quintic2}
produces $O_D(U^{1/2})$ possible $x$
for which $F(x,y)$ is a non-zero square, for each choice of $y$. We therefore obtain a bound
\beql{aaa}
\mu(U,V)\ll U^{1/2}V.
\eeq
when $Y\nmid F(X,Y)$.

In the alternative case, in which $F(X,Y)=YG(X,Y)$ say, our hypotheses show that $Y$ cannot divide
$G(X,Y)$. Moreover $f(X)=yG(X,y)$ cannot be a constant multiple of a square if $y\not=0$, since $F$ 
has no repeated factor. Thus Lemma \ref{quintic2} shows that there are
$O_D(U^{1/2})$ possible $x$ for each non-zero $y$. Since $F(x,y)$ cannot be a 
non-zero square when $y=0$ we conclude that (\ref{aaa}) holds also when $Y|F(X,Y)$.

We may of course reverse the roles of $X$ and $Y$ to give a bound 
\[\mu(U,V)\ll V^{1/2}U.\]
Comparing these estimates with Lemma \ref{quintic} we now find that
\[\mu(U,V)\ll \left\{(UV)^{2/3}+\min\big(U+V\,,\,U^{1/2}V\,,\,V^{1/2}U\big)\right\}(\log 2||F||)^2.\]
Finally we observe that
\begin{eqnarray*}
\min\big(U+V\,,\,U^{1/2}V\,,\,V^{1/2}U\big)&\le& \min\big(U\,,\,U^{1/2}V\big)+\min\big(V\,,\,V^{1/2}U\big)\\
&\le &U^{1/3}\{U^{1/2}V\}^{2/3}+V^{1/3}\{V^{1/2}U\}^{2/3}\\
&=&2(UV)^{2/3},
\end{eqnarray*}
and the lemma follows.
\end{proof}

The forms $F$ of interest to us arise as $L_1^2M_1^3+L_2^2M_2^3$, where $L_1,L_2,M_1,M_2$ are
linear forms. It turns out that Lemma \ref{q3} will be applicable, 
except in certain exceptional circumstances.

\begin{lemma}\label{except}
Let $L_1(u,v),L_2(u,v),M_1(u,v),M_2(u,v)\in\Z[u,v]$ be non-zero linear forms, and suppose that no two
are proportional, apart possibly from $L_1$ and $M_1$, or $L_2$ and $M_2$. Let $F=L_1^2M_1^3+L_2^2M_2^3$
and define $\mu(X,Y)$ as in Lemma \ref{q3}. Then if $X,Y\ge 1$ we have
\[\mu(X,Y)\ll (XY)^{2/3}\log^2(2||F||XY),\]
except when there are non-zero rationals $g$ and $\nu$ such that
\[L_1(u,v)=\nu\{4M_1(u,v)-5g^2M_2(u,v)\}\]
and
\[L_2(u,v)=\nu g^3\{5M_1(u,v)-4g^2M_2(u,v)\}.\]
\end{lemma}

\begin{proof}
This follows immediately from Lemma \ref{q3} when $F$ is square-free. If $F$ factors over $\overline{\Q}$ as 
$L^2C$ with a linear form $L$ and a square-free cubic $C$ then there must be such a factorization over $\Z$. 
Then $\mu(X,Y)$ counts non-zero square values of $C(u,v)$ such that $L(u,v)\not=0$, 
while Lemma \ref{q3}, applied to $C(u,v)$, shows that there are
\[\ll (XY)^{2/3}\log^2(2||C||XY)\]
non-zero square values of $C$. Since $C$ is a factor of $F$ over $\Z$ it follows that $||C||\ll ||F||$, so that
Lemma \ref{except} follows in this case too.

It remains to consider the situation in which $F$ factors as $Q^2L$ over  $\overline{\Q}$. As before it follows that 
we may assume $Q$ and $L$ to be defined over $\Q$.  Our conditions ensure that $F$ does not vanish 
identically.  Moreover, since $M_1$ and $M_2$ are not proportional we can write $L_1,L_2$ and $L$ in terms of
them as 
\[L_1(u,v)=aM_1(u,v)+bM_2(u,v),\;\;\; L_2(u,v)=cM_1(u,v)+dM_2(u,v),\]
and
\[L(u,v)=eM_1(u,v)+fM_2(u,v)\]
with rational coefficients $a,\ldots,f$.  Clearly $a$ and $d$ cannot vanish since neither $L_1$ and $M_2$,
nor $L_2$ and $M_1$ are proportional.
If $L$ were proportional to $L_1$, say, then $L_1$ would divide $F$ and hence also $L_2^2M_2^3$. This would 
contradict our hypotheses that $L_1$ is coprime to both $L_2$ and $M_2$. A similar argument shows that $L$ 
cannot be proportional to any of $L_1, L_2, M_1$ or $M_2$, whence each of $af-be$, $cf-de$, $e$, and $f$ must 
be non-zero.  We now choose $u_0,v_0\in\Q$ such that $M_1(u_0,v_0)=f$ and $M_2(u_0,v_0)=-e$, which is
possible since $M_1$ and $M_2$ are not proportional. Then $L(u_0,v_0)=0$, whence
\[F(u_0,v_0)=(af-be)^2f^3-(cf-de)^2e^3=0.\]
We conclude that we can write $f=g^2e$ and 
\beql{rel1}
(af-be)g^3=cf-de
\eeq
for some non-zero $g\in\Q$. Thus 
$L(u,v)=e\{M_1(u,v)+g^2M_2(u,v)\}$ and 
\begin{eqnarray*}
e\{g^3L_1(u,v)-L_2(u,v)\}&=&e(g^3a-c)M_1(u,v)+e(g^3b-d)M_2(u,v)\\
&=&(g^3a-c)\{eM_1(u,v)+fM_2(u,v)\}\\
&=&(g^3a-c)L(u,v).
\end{eqnarray*}
We then calculate that
\begin{eqnarray*}
eF&=&eL_1^2(M_1^3+g^6M_2^3)-(eg^6L_1^2-eL_2^2)M_2^3\\
&=&L_1^2(eM_1+eg^2M_2)(M_1^2-g^2M_1M_2+g^4M_2^2)\\
&&\hspace{2cm}-(g^3L_1+L_2)(eg^3L_1-eL_2)M_2^3\\
&=&L_1^2L(M_1^2-g^2M_1M_2+g^4M_2^2)-(g^3L_1+L_2)(g^3a-c)LM_2^3,
\end{eqnarray*}
whence
\begin{eqnarray*}
eQ^2&=&(aM_1+bM_2)^2(M_1^2-g^2M_1M_2+g^4M_2^2)\\
&&\hspace{3cm}-(g^3a-c)\{g^3(aM_1+bM_2)+(cM_1+dM_2)\}M_2^3\\
&=&AM_1^4+BM_1^3M_2+CM_1^2M_2^2+DM_1M_2^3+EM_2^4,
\end{eqnarray*}
say, where
\[A=a^2,\;\;\; B=2ab-a^2g^2,\;\;\; C=b^2-2abg^2+a^2g^4,\]
\[D=2abg^4-b^2g^2-(g^3a-c)(g^3a+c)=c^2-g^2(b-ag^2)^2\]
and
\[E=b^2g^4-(g^3a-c)(g^3b+d).\]
To simplify things we write
\[a=\alpha,\; b=g^2\beta,\; c=g^3\gamma,\; d=g^5\delta\]
so that the relation (\ref{rel1}) reduces to $\alpha-\beta=\gamma-\delta$. If we 
write $\mu=\alpha-\beta=\gamma-\delta$ for this difference, our formulae 
then become $A=(\beta+\mu)^2$, $B=g^2(\beta^2-\mu^2)$, $C=g^4\mu^2$, 
$D=g^6(\gamma^2-\mu^2)$ and $E=g^8(\gamma-\mu)^2$, so that
\begin{eqnarray*}
eQ^2&=&(\beta+\mu)^2M_1^4+g^2(\beta^2-\mu^2)M_1^3M_2+g^4\mu^2M_1^2M_2^2\\
&&\hspace{3cm}\mbox{}+g^6(\gamma^2-\mu^2)M_1M_2^3+g^8(\gamma-\mu)^2M_2^4.
\end{eqnarray*}
Since $M_1$ and $M_2$ are linearly independent we may write $Q(u,v)$ as $Q_*(M_1,M_2)$
so that
\begin{eqnarray*}
eQ_*(X,Y)^2&=&(\beta+\mu)^2X^4+g^2(\beta^2-\mu^2)X^3Y+g^4\mu^2X^2Y^2\\
&&\hspace{3cm}\mbox{}+g^6(\gamma^2-\mu^2)XY^3+g^8(\gamma-\mu)^2Y^4.
\end{eqnarray*}
We then see that $e$ must be a square, $e=h^2$, say, and
\[\pm hQ(X,Y)=(\beta+\mu)X^2+\kappa g^2XY+\epsilon g^4(\gamma-\mu)Y^2\]
for some coefficient $\kappa$ and some choice $\epsilon=\pm 1$. Then
\[B=2\kappa g^2(\beta+\mu)=g^2(\beta^2-\mu^2),\]
whence $\kappa=\tfrac12 (\beta-\mu)$. Note here that $\beta+\mu=\alpha=a\not=0$, as
already noted. Similarly we have 
\[D=2\epsilon\kappa g^6(\gamma-\mu)=g^6(\gamma^2-\mu^2),\]
whence $\kappa=\epsilon\tfrac12(\gamma+\mu)$. Thus $\beta=\gamma+2\mu$ for 
$\epsilon=+1$, and $\beta=-\gamma$ when $\epsilon=-1$. Moreover
\[C=g^4\mu^2=g^4\kappa^2+2\ep g^4(\beta+\mu)(\gamma-\mu),\]
so that 
\[\mu^2=\tfrac14(\gamma+\mu)^2+2(\gamma+3\mu)(\gamma-\mu)\]
when $\epsilon=+1$. However this reduces to $\gamma^2+2\gamma\mu-3\mu^2=0$.
Thus if $\epsilon=+1$ we either have $\gamma=\mu$ (giving a contradiction, since 
$\delta$ is non-zero), or $\gamma=-3\mu$ (whence $\beta=-\mu$ and $\alpha=0$,
giving another contradiction). This leaves us with the case $\epsilon=-1$, in which
\[C=g^4\mu^2=g^4\{\tfrac14(\gamma+\mu)^2-2(-\gamma+\mu)(\gamma-\mu)\}.\]
This time we conclude that $9\gamma^2-14\gamma\mu+5\mu^2=0$, so that either
$\gamma=\mu$ or $\gamma=\tfrac59\mu$. The former possibility leads to $\delta=0$,
which is impossible, leaving us with the sole example in which 
\[(\alpha,\beta,\gamma,\delta)=\frac{\mu}{9}(4,-5,5,-4).\]
This produces the exceptional case in the lemma.
\end{proof}
The reader may check that in the exceptional case one has
\[F(u,v)=\nu^2\left(4M_1^2-7g^2M_1M_2+4g^4M_2^2\right)^2\left(M_1+g^2M_2\right),\]
so that $F$ does indeed factor as $Q^2L$.
\bigskip

Our next result describes solutions of quadratic congruences.
\begin{lemma}\label{Omega}
Let $Q(u,v)\in\Z[u,v]$ be a quadratic form, and let $r\in\N$. Then there are 
lattices $\sL_1,\ldots,\sL_N\subseteq\Z^2$ with $N\le 2^{\Omega(r)}$, such that 
$Q(u,v)\equiv 0\modd{r}$ if and only if $(u,v)\in\cup_{n=1}^N\, \sL_n$.
\end{lemma}
\begin{proof}
We argue by induction on $\Omega(r)$, the case $r=1$ being
trivial. When $r=p$ is prime 
the required lattices are $\Z^2$ if $Q$ vanishes modulo $p$, or
$p\Z^2$ if $Q$ is 
irreducible modulo $p$, or $\{(u,v):p|L(u,v)\}$ and $\{(u,v):p|L'(u,v)\}$ when 
$Q$ factors as $LL'$ modulo $p$. This suffices for the case $\Omega(r)=1$.

Now suppose the result holds 
for $r$ and consider divisibility by $pr$ for some prime $p$. Suppose
that the lattices  
corresponding to $r$ are $<\mathbf{g}_n,\mathbf{h}_n>$ for $n\le N\le
2^{\Omega(r)}$,  
so that $r|Q(u,v)$ if and only if $(u,v)=x\mathbf{g}_n+y\mathbf{h}_n$
for some $x,y\in\Z$ and  
some $n\le N$. Then $Q(X\mathbf{g}_n+Y\mathbf{h}_n)$ must be
identically divisible by $r$, 
and so equal to $rQ_n(X,Y)$, say. Thus $pr|Q(u,v)$ if and only if
there is an $n$ such that 
$(u,v)=x\mathbf{g}_n+y\mathbf{h}_n$ with $p|Q_n(x,y)$. As we have
seen, this last condition  
is equivalent to at most two lattice constraints, producing at most
$2N\le 2^{\Omega(pr)}$  
lattices corresponding to the congruence $Q(u,v)\equiv
0\modd{pr}$. This completes 
the induction step, and so suffices for the lemma.
\end{proof}

We will also use the following simple result on primitive integer
points in 2-dimensional lattices, 
see the author \cite[Lemma 2]{hbDASF}, for  example.
\begin{lemma}\label{DASF}
Let $\sL\subseteq\Z^2$ be a 2-dimensional lattice of determinant
$\det(\sL)$, and let $E$ be an 
ellipse, centred on the origin, of area $A$. Then
\[\card\{(x,y)\in\sL\cap A: \,\gcd(x,y)=1\}\ll \frac{A}{\det(\sL)}+1.\]
\end{lemma}

Our final auxiliary result is a version of a ``Siegel's Lemma''
sometimes thought of 
as a  ``projective slicing'' argument. 
For a $k$-dimensional lattice $\mathsf{M}\in\R^k$ we write
\[\widehat{\mathsf{M}}=\{\mathbf{u}\in\R^k:\mathbf{u}^T\mathbf{m}\in\Z,\;
\forall \mathbf{m}\in\mathsf{M}\}.\]
\begin{lemma}\label{siegel}
  Let $\mathsf{M}\in\R^k$ be a $k$-dimensional lattice, where $k\ge 2$.  
Then there is a set
  $H\subset\mathsf{M}$ with $\card H\ll_k
1+\det(\mathsf{M})^{1/(k-1)}$
such that 
\[||\mathbf{m}||_\infty\ll_k 1+\det(\mathsf{M})^{1/(k-1)}\]
for all $\mathbf{m}\in H$,  and so that every
  $\mathbf{u}\in\widehat{\mathsf{M}}$ with $||\mathbf{u}||_\infty\le 1$ is
  orthogonal to some non-zero element of $H$.
\end{lemma}
When $\mathsf{M}=\lambda\Z^k$ for some $\lambda\ge 1$ the proof is
straightforward, but for lattices 
of general shapes more work is required.
\begin{proof}
We begin by dealing with the easy case in which $\det(\mathsf{M})<k^{-k}$.
By Minkowski's Theorem there is a shortest non-zero vector
$\mathbf{m}_1\in\mathsf{M}$ with
\[||\mathsf{m}_1\||_\infty\le \det(\mathsf{M})^{1/k}<k^{-1}.\]
Then if $\mathbf{u}\in\widehat{\mathsf{M}}$ with
$||\mathbf{u}||_\infty\le 1$ we have 
$\mathsf{u}^T\mathsf{m}_1\in\Z$ and $|\mathsf{u}^T\mathsf{m}_1|<1$, so that
$\mathsf{u}^T\mathsf{m}_1=0$. We may therefore take
$H=\{\mathbf{m}_1\}$ in this case. 

For the rest of the proof we may assume that $\det(\mathsf{M})\ge k^{-k}$.
Let $\mathcal{C}=[-1,1]^k$, and write
$\mathcal{N}(t)=\card\left(t\CC\cap\mathsf{M}\right)$. 
We then claim that every
  $\mathbf{u}\in\widehat{\mathsf{M}}$ with $||\mathbf{u}||_\infty\le 1$ is
  orthogonal to some non-zero element of $H=2t\CC\cap\mathsf{M}$, provided that
we have $\mathcal{N}(t)>1+2kt$. To prove this we note that each vector $\mathbf{m}$ in 
$t\CC\cap\mathsf{M}$ has 
$\mathbf{u}^T\mathbf{m}\in\Z$, and $|\mathbf{u}^T\mathbf{m}|\le kt$, 
so that there are at most $1+2kt$ possible values for
$\mathbf{u}^T\mathbf{m}$. Thus if $\mathcal{N}(t)>1+2kt$ there will be distinct
vectors $\mathbf{m}_1$ and $\mathbf{m}_2$ with
$\mathbf{u}^T\mathbf{m}_1=\mathbf{u}^T\mathbf{m}_2$. It then follows that
$\mathbf{u}^T(\mathbf{m}_1-\mathbf{m}_2)=0$, with
$0<||\mathbf{m}_1-\mathbf{m}_2||_\infty\le 2t$.  Since
$\mathbf{m}_1-\mathbf{m}_2\in 2t\CC\cap\mathsf{M}$ the claim follows.

To investigate the size of $\mathcal{N}(t)$ we use Lemma 1, 
part (iii), of the author's work \cite{annals}. This shows that there is a basis 
$\mathbf{b}^{(1)},\ldots,\mathbf{b}^{(k)}$ of $\mathsf{M}$ such that
\beql{ins}
 c_1\prod_{i=1}^k(1+t/||\mathbf{b}^{(i)}||_\infty)\le \mathcal{N}(t)\le
c_2\prod_{i=1}^k(1+t/||\mathbf{b}^{(i)}||_\infty)
\eeq
and
\beql{db}
\prod_{i=1}^k||\mathbf{b}^{(i)}||_\infty \le c_3\det(\mathsf{M}),
\eeq
with positive constants $c_1,c_2,c_3$ depending only on $k$.
Since $k\ge 2$ we may define $t_0$ as the least positive number for which
\beql{ins1}
c_1\prod_{i=1}^k(1+t_0/||\mathbf{b}^{(i)}||_\infty)=1+c_1+2kt_0.
\eeq
Then $\mathcal{N}(t_0)>1+2kt_0$ by (\ref{ins}), so that we may take
$H=2t_0\CC\cap\mathsf{M}$ in the above. It follows from (\ref{db}) and 
(\ref{ins1}) that
\[\frac{t_0^k}{\det(\mathsf{M})}\ll_k
\frac{t_0^k}{\prod_{i=1}^k||\mathbf{b}^{(i)}||_\infty}\ll_k 1+t_0,\]
whence
\[t_0\ll_k \det(\mathsf{M})^{1/k}+\det(\mathsf{M})^{1/(k-1)}\ll_k \det(\mathsf{M})^{1/(k-1)},\]
since we are assuming that $\det(\mathsf{M})\ge k^{-k}$. 
Thus to conclude the proof of the lemma it only remains to observe that
\[\card H=\mathcal{N}(2t_0)\ll_k \prod_{i=1}^k(1+2t_0/||\mathbf{b}^{(i)}||_\infty)
\ll_k 1+t_0\]
by (\ref{ins}) and (\ref{ins1}), and that every $\mathbf{m}\in H$ has 
$||\mathbf{m}||_\infty\le 2t_0$. 
\end{proof}

\section{Proof of Theorem \ref{thm2} --- Easier Bounds}
For the proof of Theorem \ref{thm2} we begin by splitting the variables $x_i,y_i$ into dyadic ranges
\beql{dyad}
\tfrac12 X_i<x_i\le X_i\;\;\;\mbox{and}\;\;\; \tfrac12 Y_i<y_i\le Y_i\;\;(i=1,2,3),
\eeq
with the $X_i$ and $Y_i$ running over powers of 2 up to $2B$ at most. In view of (\ref{range}) we may assume that
\beql{range1}
X_i^2Y_i^3\le 32B/|a_i|\;\;\;(i=1,2,3).
\eeq
Thus there is a choice of 
$X_1,X_2,X_3,Y_1,Y_2,Y_3$ such that
\beql{0}
n(B;a_1,a_2,a_3)\ll (\log B)^6\card S_0(B;X_1,X_2,X_3,Y_1,Y_2,Y_3),
\eeq
where $S_0(B;X_1,X_2,X_3,Y_1,Y_2,Y_3)$ is the set of points $(x_1,x_2,x_3,y_1,y_2,y_3)$
lying in the ranges (\ref{dyad}), satisfying (\ref{axy}) and (\ref{gcd}), and
with $a_1y_1,a_2y_2$ and $a_3y_3$ square-free. The parameters $X_i,Y_i$ will be fixed throughout
our argument, so we will write $S_0(B)$ for brevity in place of
$S_0(B;X_1,X_2,X_3,Y_1,Y_2,Y_3)$, and define $N_0(B)=\card S_0(B)$.
The aim of this section is to give some relatively easy bounds for $N_0(B)$.

Our first result is the following.
\begin{lemma}\label{YYY}
For any fixed $\ep>0$ we have
\beql{XXXf}
N_0(B)\ll_\ep \{Y_1Y_2Y_3+(X_1X_2X_2)^{1/3}\}B^\ep \ll \{Y_1Y_2Y_3+B^{1/2}\}B^\ep.
\eeq
Indeed, either $N_0(B)\ll_\ep B^{1/2+\ep}$ or each triple $(y_1,y_2,y_3)$ contributes 
$O_\ep(B^\ep)$ to $N_0(B)$.
\end{lemma}
\begin{proof}
The proof depends on Corollary 2 of Browning and Heath-Brown \cite{BHB}. This applies to a general
integral ternary quadratic form $Q$ of matrix $\mathbf{M}$. One writes $\Delta=|$det$(\mathbf{M})|$, and takes 
$\Delta_0$ to be highest common factor of the $2\times 2$ minors of $\mathbf{M}$. The result is then 
that the number of integral solutions $Q(x_1,x_2,x_3)=0$ with $|x_i|\le X_i$ for $i=1,2,3$ is
\[\ll \left\{1+\left(\frac{X_1X_2X_3\Delta_0^{3/2}}{\Delta}\right)^{1/3}\right\}d(\Delta).\]

We apply this to the form with matrix diag($a_1y_1^3,a_2y_2^3,a_3y_3^3)$, which has 
$Y_1Y_2Y_3\ll \Delta\ll B^3$ and $\Delta_0=1$. Thus each triple $(y_1,y_2,y_3)$ contributes
\beql{ind}
\ll_\ep \left\{1+\frac{(X_1X_2X_3)^{1/3}}{Y_1Y_2Y_3}\right\}B^\ep
\eeq
to $N_0(B)$. The first estimate of (\ref{XXXe}) then follows on summing over $y_1,y_2$ and $y_3$.
We then note that $X_1,X_2,X_3\ll B^{1/2}$ by (\ref{range1}), whence $(X_1X_2X_3)^{1/3}\ll B^{1/2}$.
This gives us the second bound of (\ref{XXXf}).

Finally, the estimate (\ref{ind}) is $O_\ep(B^\ep)$ unless $Y_1Y_2Y_3\le (X_1X_2X_3)^{1/3}$. In the 
latter case however we have $Y_1Y_2Y_3\ll B^{1/2}$, since (\ref{range1}) shows that $X_i\ll B^{1/2}$
for each index $i$. Thus either the bound (\ref{ind}) is $O_\ep(B^\ep)$ or (\ref{XXXf}) will yield
$N_0(B)\ll_\ep B^{1/2+\ep}$. This completes the proof of the lemma.
\end{proof}

We also have the following estimate.
\begin{lemma}\label{XXY}
If $\{i,j,k\}=\{1,2,3\}$ and $x_i, x_j$ and $y_k$ are given, the number of triples $(x_k,y_i,y_j)$ satisfying 
the conditions of Theorem \ref{thm2} is
\[\ll_\ep \left\{1+(Y_iY_j)^{2/3}Y_k^{-2}\right\} B^\ep\]
for any $\ep>0$.
\end{lemma}
\begin{proof}
Without loss of generality we may take $i=1$, $j=2$, and $k=3$. 
The congruence $a_1x_1^2+a_2x_2^2\lambda^3\equiv 0\modd{y_3^3}$ has 
$O(3^{\omega(y_3)})$ solutions $\lambda\modd{y_3^3}$, so that the pair $(y_1,y_2)$ lies on one of 
$O_\ep(B^\ep)$ lattices $\mathsf{\Lambda}\subseteq\Z^2$ of the shape
\[\mathsf{\Lambda}=\{(u_1,u_2)\in\Z^2: u_2\equiv \lambda u_1\modd{y_3^3}\}.\]
We proceed to count solutions $(x_3,y_1,y_2)$ of (\ref{axy}) for which $(y_1,y_2)$ lies on some particular lattice
$\mathsf{\Lambda}$. We write $D=$Diag$(Y_1,Y_2)$, so that $||D^{-1}(y_1,y_2)||_\infty\le 1$. Moreover
$D^{-1}(y_1,y_2)$ lies in the lattice $D^{-1}\mathsf{\Lambda}$, which will have determinant 
$Y_1^{-1}Y_2^{-1}y_3^3$. We now choose a basis $\mathbf{g},\mathbf{h}$ for $D^{-1}\mathsf{\Lambda}$
with 
\[||\mathbf{g}||_\infty=\lambda_1\;\;\;\mbox{and}\;\;\; ||\mathbf{h}||_\infty=\lambda_2,\]
where $\lambda_1,\lambda_2$ 
are the successive minima of $D^{-1}\mathsf{\Lambda}$ for the $||\cdot||_\infty$ norm.  Then
any $(v_1,v_2)\in D^{-1}\mathsf{\Lambda}$ may be written as $u\mathbf{g}+v\mathbf{h}$ for integers $u$ 
and $v$ satisfying
\[ u\ll \frac{||u\mathbf{g}+v\mathbf{h}||_\infty}{\lambda_1}\;\;\;\mbox{and}\;\;\;
v\ll \frac{||u\mathbf{g}+v\mathbf{h}||_\infty}{\lambda_2}.\]
In particular, $D^{-1}(y_1,y_2)=u\mathbf{g}+v\mathbf{h}$ with $|u|\le U$ and $|v|\le V$, where $U\ll\lambda_1^{-1}$
and $V\ll\lambda_2^{-1}$, so that
\[UV\ll (\lambda_1\lambda_2)^{-1}\ll \det(D^{-1}\mathsf{\Lambda})^{-1}=Y_1Y_2y_3^{-3}.\]
We then have $(y_1,y_2)=uD\mathbf{g}+vD\mathbf{h}$, where $D\mathbf{g}$ and $D\mathbf{h}$ are in
$\mathsf{\Lambda}$, and so are integer vectors. We therefore conclude that there are integral linear forms
$L_1(u,v)$ and $L_2(u,v)$ such that $y_1=L_1(u,v)$ and $y_2=L_2(u,v)$. Since $\mathbf{g}$ and 
$\mathbf{h}$ are linearly independent, so are $L_1$ and $L_2$.

We now define 
\[F(X,Y)=-a_3y_3\{a_1x_1^2L_1(X,Y)^3+a_2x_2^2L_2(X,Y)^3\},\]
so that $F(u,v)=(a_3y_3^2x_3)^2$ is a non-zero square whenever $y_1=L_1(u,v)$ and $y_2=L_2(u,v)$
for a solution of (\ref{axy}). Moreover, $F$ cannot have a repeated factor since $X^3+Y^3$ does not, and 
$L_1$ and $L_2$ are linearly independent. We may therefore apply Lemma \ref{q3}, which shows that
there are $O((UV)^{2/3}(\log B)^2)$ relevant pairs $y_1,y_2$. Since the number of possible $\mathsf{\Lambda}$
is $O_\ep(B^\ep)$ we conclude that there are 
\[\ll_\ep (UV)^{2/3}(\log B)^2B^\ep\ll_\ep (Y_1Y_2)^{2/3}Y_3^{-2}B^{2\ep}\]
triples $(x_3,y_1,y_2)$, provided that $U,V\ge 1$. However if $U<1$, say, then $u=0$ and so $y_1$ and 
$y_2$ cannot  be coprime unless $v=\pm 1$. Thus there is at most one suitable set of positive integer 
values $y_1,y_2,x_3$  in this case, and similarly if $V<1$. The lemma then follows on re-defining $\ep$.
\end{proof}

Lemmas \ref{YYY} and  \ref{XXY} have the following corollary.
\begin{lemma}\label{CXXY}
Let $i,j,k$ be a permutation of the indices $1,2,3$, and let $\ep>0$ be given. Then
\beql{XXXe}
N_0(B)\ll_\ep  \{X_iX_jY_k+B^{6/11}\}B^\ep.
\eeq
\end{lemma}
\begin{proof}
It follows immediately from Lemma \ref{XXY} that
\[N_0(B)\ll_\ep  \{X_iX_jY_k+X_iX_j(Y_iY_j)^{2/3}Y_k^{-1}\}B^\ep.\]
Thus either $N_0(B)\ll_\ep  X_iX_jY_kB^\ep$ or $N_0(B)\ll_\ep X_iX_j(Y_iY_j)^{2/3}Y_k^{-1}B^\ep$.
In the latter case we see from Lemma \ref{YYY} and (\ref{range1}) that either $N_0(B)\ll_\ep B^{1/2+\ep}$ or
\begin{eqnarray*}
N_0(B)&\ll_\ep & \min\{Y_1Y_2Y_3\,,\,X_iX_j(Y_iY_j)^{2/3}Y_k^{-1}\}B^\ep\\
&\le & \{Y_1Y_2Y_3\}^{5/11}\{X_iX_j(Y_iY_j)^{2/3}Y_k^{-1}\}^{6/11}B^\ep\\
&=& (X_i^2Y_i^3)^{3/11}(X_j^2Y_j^3)^{3/11}Y_k^{-1/11}B^\ep\\
&\ll& B^{6/11+\ep},
\end{eqnarray*}
which suffices for the lemma.
\end{proof}

Finally we show that the only case preventing one getting an exponent less than $3/5$ in Theorem
\ref{thm:main} is that in which all the variables are around $B^{1/5}$.
\begin{lemma}\label{crit}
Let $\delta>0$. Then we have
\[N_0(B)\ll_\ep B^{6/11+\ep}+B^{3/5-\delta+\ep}\]
for any fixed $\ep>0$, except possibly when
\[X^{1/5-3\delta/2}\le X_k\le X^{1/5+3\delta/2}\;\;\;\mbox{and}\;\;\; 
X^{1/5-\delta}\le Y_k\le X^{1/5+\delta},\;\;(k=1,2,3).\]
\end{lemma}
\begin{proof}
From Lemmas \ref{YYY} and \ref{CXXY} we see that either $N_0(B)\ll_\ep B^{6/11+\ep}$ or
\begin{eqnarray*}
N_0(B)&\ll_\ep &\min\big(Y_1Y_2Y_3\,,\,X_iX_jY_k\big)B^{\ep}\\
&\le &(Y_1Y_2Y_3)^{3/5}(X_iX_jY_k)^{2/5}B^{\ep}\\
&=& Y_k(X_i^2Y_i^3)^{1/5}(X_j^2Y_j^3)^{1/5}B^\ep\\
&\ll& Y_kB^{2/5+\ep}
\end{eqnarray*}
for any ordering $i,j,k$ of the indices $1,2,3$, by (\ref{range1}). This proves the lemma when 
$Y_k\le B^{1/5-\delta}$; and if $X_k\ge B^{1/5+3\delta/2}$
we find from (\ref{range1}) that $Y_k\ll B^{1/5-\delta}$, so that the lemma follows in this case too.

Similarly, from Lemmas \ref{YYY} and \ref{CXXY} we see that either $N_0(B)\ll_\ep B^{6/11+\ep}$ or
\begin{eqnarray*}
N_0(B)&\ll_\ep &\min\big(Y_1Y_2Y_3\,,\,X_iX_jY_k\,,\, X_iX_kY_j\big)B^{\ep}\\
&\le &(Y_1Y_2Y_3)^{1/5}(X_iX_jY_k)^{2/5}(X_iX_kY_j)^{2/5}B^{\ep}\\
&=& X_i^{2/3}(X_i^2Y_i^3)^{1/15}(X_j^2Y_j^3)^{1/5}(X_k^2Y_k^3)^{1/5}B^\ep\\
&\ll& X_j^{2/3}B^{7/15+\ep}.
\end{eqnarray*}
This is sufficient for the lemma when $X_j\le B^{1/5-3\delta/2}$; and if $Y_j\ge B^{1/5+\delta}$
 we find from (\ref{range1}) that $X_j\ll B^{1/5-3\delta/2}$, whence the lemma follows in this case too.
 \end{proof}
 
 In light of Lemma \ref{crit} it will suffice for Theorem \ref{thm2} to handle the case in which
\beql{range2}
\begin{array}{c}
X^{1/5-1/60}\le X_k\le X^{1/5+1/60},\\
\rule{0cm}{5mm}X^{1/5-1/90}\le Y_k\le X^{1/5+1/90},\end{array} \;\;\; (k=1,2,3),
\eeq
and we therefore assume these bounds for the rest of the paper.

\section{Covering by Lines}

We can now embark on the proof of Theorem \ref{thm2}.  In following the argument the reader may 
find it helpful to think in particular about the situation in
which $|a_1|=|a_2|=|a_3|=1$ and 
\[B^{1/5}\ll X_1=X_2=X_3=Y_1=Y_2=Y_3\ll B^{1/5}.\]
Given Lemma \ref{crit} we shall refer to this as the ``Critical Case''.
\bigskip

The key idea is to consider  
the equation (\ref{axy}) as a congruence to modulus $y_3^3$.
Thus $y_3$ will play a special role, and we emphasize this by using
the notation $y_3=q$, which is more 
suggestive of the modulus of a congruence.  We proceed to classify our solutions
further, according to the value of $q=y_3$. Moreover, since
$x_1,x_2,y_1$ and $y_2$ are all coprime to $q$, we can
classify the solutions into smaller classes still, according to the value of
$\kappa\modd{q}$ for which
\beql{lamd}
x_1 y_1+ \kappa x_2 y_2\equiv 0\modd{q}.
\eeq
Throughout the rest of the paper, whenever we use the notations $q$
and $\kappa$ we shall tacitly 
assume that $\tfrac12 Y_3<q\le Y_3$ with $q$ square-free and coprime
to $a_1a_2$, 
and that $\kappa$ is taken modulo $q$, with $\gcd(q,\kappa)=1$.
We will write 
\[S(q,\kappa)=S(q,\kappa;X_1,X_2,X_3,Y_1,Y_2)\]
for the set of integral 4-tuples
$(x_1,x_2,y_1,y_2)$ satisfying (\ref{lamd}) for given values of $q$
and $\kappa$, such that 
$(x_1,x_2,x_3,y_1,y_2,q)$ is in $S_0(B)$ for some choice of $x_3$.  As
a preliminary 
observation we note that, since $q|a_1x_1^2y_1^3+a_2x_2^2y_2^3$ and 
$a_1,a_2,x_1,x_2,y_1$ and $y_2$ are all coprime to $q$, the condition
(\ref{lamd}) implies that 
\beql{c1}
a_2x_1\equiv \kappa^3 a_1x_2\modd{q}\;\;\; \mbox{and}\;\;\;
\kappa^2 a_1y_1+a_2 y_2\equiv 0\modd{q}.
 \eeq
 
We begin the main  argument with the following lemma.
\begin{lemma}\label{TheLattice}
For every $(x_1,x_2,y_1,y_2)\in S(q,\kappa)$ we have
\beql{c2}
\kappa^2a_1x_2y_1+ a_2x_2y_2\equiv 0\modd{q},
\eeq
\beql{c3}
2a_2x_1y_1+\kappa^3a_1x_2y_1+3\kappa a_2 x_2y_2\equiv 0\modd{q^2},
\eeq
and
\beql{c4}
5\kappa^2a_1a_2x_1y_1+a_2^2x_1y_2+\kappa^5a_1^2x_2y_1+5\kappa^3a_1a_2x_2y_2
\equiv 0\modd{q^3}.
\eeq

On the other hand, if $(x_1,x_2,y_1,y_2)$ satisfies the congruences
(\ref{c1}), (\ref{c2}), (\ref{c3}) and (\ref{c4}), and 
$\gcd(y_1,y_2,q)=r$, then
\beql{c5}
18(a_1x_1^2y_1^3+a_2x_2^2y_2^3)\equiv 0\modd{q^3r^2}.
\eeq
\end{lemma}
\begin{remark*}
Although $\kappa$ was only defined modulo $q$, the
congruences (\ref{c3}) and (\ref{c4}) 
none the less hold modulo $q^2$ and $q^3$ respectively, irrespective
of the choice of $\kappa$. 
\end{remark*}

\begin{proof}
We obtain (\ref{c2}) trivially by multiplying the second of the
congruences (\ref{c1}) by $x_2$. 
For the other claims we begin by writing
\beql{uvd}
u_1=a_2 x_1,\; u_2=\kappa^3 a_1 x_2,\; v_1=\kappa^2 a_1 y_1,\;
v_2=a_2 y_2
\eeq
so that (\ref{c1}) produces $u_1\equiv u_2\modd{q}$ and
$v_1\equiv -v_2\modd{q}$. 
Moreover
\beql{uv}
u_1^2v_1^3+u_2^2v_2^3=
\kappa^6 a_1^2a_2^2\left(a_1x_1^2y_1^3+a_2x_2^2y_2^3\right)\equiv
0\modd{q^3}.
\eeq
We may now choose $\alpha$ and $\beta$, both divisible by $q$, such
that
\[u_1\equiv u_2(1+\alpha)\modd{q^3}\;\;\;\mbox{and}\;\;\;
v_2\equiv -v_1(1+\beta)\modd{q^3}.\]
The congruence (\ref{uv}) implies that
$(1+\alpha)^2\equiv(1+\beta)^3\modd{q^3}$, so that
\beql{albe}
2\alpha+\alpha^2\equiv 3\beta+3\beta^2+\beta^3
\equiv 3\beta+3\beta^2\modd{q^3}.
\eeq
In particular we have 
\beql{23}
2\alpha\equiv 3\beta\modd{q^2},
\eeq
whence
$2\alpha^2\equiv 3\alpha\beta\modd{q^3}$ and
$3\beta^2\equiv 2\alpha\beta\modd{q^3}$. It then follows from
(\ref{albe}) that
\[4\alpha+3\alpha\beta\equiv 4\alpha+2\alpha^2\equiv
6\beta+6\beta^2\equiv 6\beta+4\alpha\beta\modd{q^3},\]
and finally that $4\alpha\equiv 6\beta+\alpha\beta\modd{q^3}$. We
multiply by $u_2v_1$, noting that $u_2\alpha\equiv u_1-u_2\modd{q^3}$
and $v_1\beta\equiv -v_1-v_2\modd{q^3}$. This produces the congruence
$4v_1(u_1-u_2)\equiv 6u_2(-v_1-v_2)+(u_1-u_2)(-v_1-v_2)\modd{q^3}$, 
which simplifies to $5u_1v_1+u_1v_2+u_2v_1+5u_2v_2\equiv 0\modd{q^3}$.
Referring to (\ref{uvd}) we now obtain (\ref{c4}).

Similarly, the congruence (\ref{23}) yields 
$2v_1(u_1-u_2)\equiv 3u_2(-v_1-v_2)\modd{q^2}$, which simplifies to
$2u_1v_1+u_2v_1+3u_2v_2\equiv 0\modd{q^2}$, giving us (\ref{c3}).
\bigskip

To handle the remaining claim in the lemma we retain the notation (\ref{uvd}),
whence $r|v_1$ in particular. Moreover, the relations (\ref{c1}),
(\ref{c3}) and (\ref{c4}) become 
\[u_1\equiv u_2\modd{q},\;\;\; v_1+v_2\equiv 0\modd{q},\]
\beql{uv1}
2u_1v_1+u_2v_1+3u_2v_2\equiv 0\modd{q^2},
\eeq
and
\beql{uv2}
5u_1v_1+u_1v_2+u_2v_1+5u_2v_2\equiv 0\modd{q^3},
\eeq
while the claim that $q^3r^2|18(a_1x_1^2y_1^3+a_2x_2^2y_2^3)$ is equivalent to
the statement that $q^3r^2|18(u_1^2v_1^3+u_2^2v_2^3)$.

Since $q|u_1-u_2$ we may write 
$u_1=u_2+qu$ for some integer $u$, and similarly we may take
$v_2=qv-v_1$.  Then (\ref{uv1}) 
becomes 
\beql{uv1'}
A:=2uv_1+3vu_2\equiv 0\modd{q}, 
\eeq
while (\ref{uv2}) reduces to
\beql{uv2'}
B:=4uv_1+6vu_2+quv\equiv 0\modd{q^2}.
\eeq
Then, if
\[C:=18(u_1^2v_1^3+u_2^2v_2^3)=18\{(u_2+qu)^2v_1^3+u_2^2(qv-v_1)^3\},\]
we may calculate that
\[C=9qu_2v_1^2B+8q^3u^2vv_1^2-8q^3uvv_1A+21q^2uv_1^2A+2q^3vA^2-6q^2v_1A^2.\]
It then follows that $q^3r^2|C$, since $q|A$, $q^2|B$, and $r|v_1$.
\end{proof}

The rest of this section is devoted to the proof of the following lemma.
\begin{lemma}\label{lemma:qkappa}
There is a set $S_0(q,\kappa)\subseteq S(q,\kappa)$ and a collection
$\mathcal{L}(q,\kappa)$ of  
two-dimensional lattices $\mathsf{M}\subset\Z^4$
with 
\beql{S0qk}
\card S_0(q,\kappa)\ll_\ep (Y_1Y_2)^{-1/3}Y_3^{-2}B^{2/3+\ep}
\eeq
and
\[\card \mathcal{L}(q,\kappa)\ll_\ep (Y_1Y_2)^{-1/3}Y_3^{-2}B^{2/3+\ep}\]
for any fixed $\ep>0$, such that
\[S(q,\kappa)\setminus
S_0(q,\kappa)\subseteq\bigcup_{\mathsf{M}\in\mathcal{L}(q,\kappa)}\mathsf{M},\]
and with the property that the congruences (\ref{c1}), (\ref{c2}),
(\ref{c3}), and (\ref{c4}) all hold for 
any $(x_1,x_2,y_1,y_2)\in\mathsf{M}$, for every
$\mathsf{M}\in\mathcal{L}(q,\kappa)$. 

For each lattice $\mathsf{M}$ there is always at least one point
$(x_1,x_2,y_1,y_2)\in \mathsf{M}$  
where $a_1x_1^2y_1^3+a_2x_2^2y_2^3\not=0$. Moreover, if
$(x_1,x_2,y_1,y_2)$ is any non-zero point  
in $\mathsf{M}$ we will have $(x_1,x_2)\not=(0,0)$ and $(y_1,y_2)\not=(0,0)$.
\end{lemma}
\begin{remark*}\label{Rqk}
In the Critical Case we will have $O(B^{2/5})$ pairs $(q,\kappa)$ and
\[\card S_0(q,\kappa)\ll_\ep B^{2/15+\ep},\]
giving a satisfactory contribution $O_\ep(B^{8/15+\ep})$ to
$N_0(B)$. 
\end{remark*}

\begin{proof}
Let $\sL\subseteq q^{-3}\Z^4$ be the lattice generated by 
\[(0,0,1,0),\;\;q^{-1}\left(0,0,\kappa^2a_1\overline{a_2},1\right),\;\; q^{-2}
\left(2,0,\kappa^3a_1\overline{a_2},3\kappa\right),\]
and
\[q^{-3}\left(5\kappa^2a_1\overline{a_2},1,
\kappa^5a_1^2\overline{a_2}^2,5\kappa^3a_1\overline{a_2}\right),\]
where $\overline{a_2}\in[1,q^3]$ denotes the inverse of $a_2$ modulo $q^3$. Then
$\det(\sL)=2q^{-6}$. Moreover if $\mathbf{g}$ is any of the above
four generators the inner product $\mathbf{w}^T\mathbf{g}$ will be an
integer for all $\mathbf{w}=(x_1y_1,x_1y_2,x_2y_1,x_2y_2)$ with
$(x_1,x_2,y_1,y_2)\in S(q,\kappa)$.  
It follows that $\mathbf{w}^T\mathbf{g}\in\Z$ for every
$\mathbf{g}\in\sL$, so that $\mathbf{w}\in\widehat{\sL}$, in the
notation of Lemma \ref{siegel}. 
If
\[E=\mathrm{Diag}\left(X_1Y_1,X_1Y_2,X_2Y_1,X_2Y_2\right)\]
we see that
\[E^{-1}\mathbf{w}=\left(\frac{x_1y_1}{X_1Y_1}\,,\,\frac{x_1y_2}{X_1Y_2}\,,\,
\frac{x_2y_1}{X_2Y_1}\,,\,\frac{x_2y_2}{X_2Y_2}\right),\]
whence $||E^{-1}\mathbf{w}||_\infty\le 1$.  Moreover we will have
$E^{-1}\mathbf{w}\in \widehat{E\sL}$. Since
$\det(E\sL)=2(X_1X_2Y_1Y_2)^2q^{-6}$ we may  
now conclude from Lemma \ref{siegel} that there is some set
$H\subseteq E\sL$ with  
\[\card H\ll 1+(X_1X_2Y_1Y_2)^{2/3}q^{-2}\ll 1+(X_1X_2Y_1Y_2)^{2/3}Y_3^{-2}\]
such that
$(E^{-1}\mathbf{w})^T\mathbf{h}=0$ for a suitable non-zero
$\mathbf{h}\in H$.  According to (\ref{range1}) we have 
\beql{range3}
X_1X_2\ll (Y_1Y_2)^{-3/2}B,
\eeq
whence
\[\card H\ll 1+(Y_1Y_2)^{-1/3}Y_3^{-2}B^{2/3}\ll (Y_1Y_2)^{-1/3}Y_3^{-2}B^{2/3},\]
by (\ref{range2}).  Since $(E^{-1}\mathbf{w})^T\mathbf{h}=0$ we have
 $\mathbf{w}^T\mathbf{g}=0$ 
 with
 \[\mathbf{g}=E^{-1}\mathbf{h}\in E^{-1}H\subseteq\sL\subseteq q^{-3}\Z^4.\] 
 The set $E^{-1}H$ has cardinality $O((Y_1Y_2)^{-1/3}Y_3^{-2}B^{2/3})$
allowing us to 
conclude that there is a set of $O((Y_1Y_2)^{-1/3}Y_3^{-2}B^{2/3})$ non-zero
 integer vectors $\mathbf{t}\in\Z^4$ such that every quadruple
 $(x_1,x_2,y_1,y_2)$ 
 in $S(q,\kappa)$ satisfies the equation
 \beql{seq}
 t_1x_1y_1+t_2x_1y_2+t_3x_2y_1+t_4x_2y_2=0
 \eeq
for some $\mathbf{t}$. According to Lemma \ref{siegel}, we have
$||\mathbf{h}||_\infty\ll (X_1X_2Y_1Y_2)^{2/3}q^{-2}$ for every
element of $H$, whence 
also $||\mathbf{g}||_\infty\ll (X_1X_2Y_1Y_2)^{2/3}q^{-2}$, since
$X_1,X_2,Y_1,Y_2$ are  
all at least 1. Hence $||\mathbf{t}||_\infty\ll
(X_1X_2Y_1Y_2)^{2/3}q\ll B^4$, say. 

We have therefore to investigate solutions to the equation
(\ref{seq}). We begin by disposing 
of the case in which $t_1t_4=t_2t_3$, which will produce the set
$S_0(q,\kappa)$. 
Suppose for example that $t_1\not=0$. Then
\[0=t_1(t_1x_1y_1+t_2x_1y_2+t_3x_2y_1+t_4x_2y_2)=(t_1x_1+t_3x_2)(t_1y_1+t_ 2y_2).\]
Thus either $t_1x_1=-t_3x_2$ or $t_1y_1=-t_2y_2$. Since we are assuming that 
$\gcd(x_1,x_2)=\gcd(y_1,y_2)=1$ it follows that $\mathbf{t}$
determines either the pair $(x_1,x_2)$ or the pair  
$(y_1,y_2)$. When $t_1t_4=t_2t_3$ but $t_1=0$, there will be some other
$t_i$ which is  
non-vanishing, and a similar argument applies. When it is $(x_1,x_2)$
that are determined by  
$q,\kappa$ and $\mathbf{t}$ we have
\[(y_1\overline{y_2})^3\equiv
-\overline{a_1}a_2\overline{x_1}^2x_2^2\modd{q^3},\] 
where, as before,  $\overline{n}\in[1,q^3]$ denotes the inverse of $n$
modulo $q^3$. Thus  
we have $y_1\equiv\lambda y_2\modd{q^3}$ where $\lambda$ is one of the
$O(3^{\omega(q)})=O_\ep(B^\ep)$ cube roots of
$-\overline{a_1}a_2\overline{x_1}^2x_2^2$  
modulo $q^3$.
The constraint $y_1\equiv\lambda y_2\modd{q^3}$ produces a two
dimensional lattice of  
determinant $q^3$, so that there are $O(1+Y_1Y_2q^{-3})$ coprime
solutions $(y_1,y_2)$ in 
the rectangle $0<y_1\le Y_1$, $0<y_2\le Y_2$, by Lemma \ref{DASF}.
Since $\tfrac12 Y_3<q=y_3\le Y_3$ we see from (\ref{range2}) that there
are $O(1)$ solutions $(y_1,y_2)$ for each $\lambda$.
This therefore produces a contribution 
\[\ll_\ep (Y_1Y_2)^{-1/3}Y_3^{-2}B^{2/3+\ep}\]
to $S_0(q,\kappa)$ from the case in which $t_1t_4=t_2t_3$
and it is $(x_1,x_2)$ that are determined by (\ref{seq}). There is a
similar argument when it 
is the values of $y_1,y_2$ which are determined, giving an overall estimate
\[\ll_\ep  (Y_1Y_2)^{-1/3}Y_3^{-2}B^{2/3+\ep}\]
for $\card S_0(q,\kappa)$, corresponding to the case in which $t_1t_4=t_2t_3$.

Turning now to the case in which $\Delta:=t_1t_4-t_2t_3\not=0$ we see
from (\ref{seq}) that 
\[y_1(t_1x_1+t_3x_2)=-y_2(t_2x_1+t_4x_2).\]
Since $y_1$ and $y_2$ are coprime we deduce that 
\beql{heq}
hy_1=-t_2x_1-t_4x_2\;\;\;\mbox{and}\;\;\; hy_2=t_1x_1+t_3x_2,
\eeq
for some $h\in\N$. Then $h|\Delta x_1$ and
$h|\Delta x_2$, whence $h|\Delta$, since $x_1$ and $x_2$ are coprime.
Thus, given $\mathbf{t}$ with $\Delta\not=0$, there are
$O_\ep(B^{\ep})$ possible values for  
$h$, each producing a line in $\mathbb{P}^3$, given by
(\ref{heq}). The corresponding integral points 
$(x_1,x_2,y_1,y_2)$ then lie on a two-dimensional lattice,
$<\mathbf{g},\mathbf{h}>$ say, in $\Z^4$.

We proceed to show that the conditions (\ref{c1}), (\ref{c2}),
(\ref{c3}), and (\ref{c4}) restrict us to a  
small number of sublattices of $<\mathbf{g},\mathbf{h}>$, which we
produce by a process of  
successive refinements. It may happen that all the points
$u\mathbf{g}+v\mathbf{h}$ already satisfy  
(\ref{c1}), but if not this condition will restrict $(u,v)$ to lie on
a sublattice of $\Z^2$, so that  
$(x_1,x_2,y_1,y_2)$ will run over a sublattice,
$<\mathbf{g}',\mathbf{h}'>$ say, of  
$<\mathbf{g},\mathbf{h}>$. The condition (\ref{c2}) is a trivial
consequence of (\ref{c1}). For (\ref{c3}) 
we write
\[(x_1,x_2,y_1,y_2)=u\mathbf{g}'+v\mathbf{h}'=
\left(x_1(u,v),x_2(u,v),y_1(u,v),y_2(u,v)\right),\] 
say, so that (\ref{c3}) corresponds to a quadratic congruence
$Q(u,v)\equiv 0\modd{q^2}$. 
We now apply Lemma \ref{Omega} with $r=q^2$. Since $q$ is square-free,
each line (\ref{heq}) produces at most 
$2^{2\omega(q)}$ two-dimensional lattices containing the relevant
$(u,v)$, and on which (\ref{c1}), 
(\ref{c2}) and (\ref{c3}) all hold. We can add the condition
(\ref{c4}) in the same way, producing a total 
of at most $2^{2\omega(q)}\times 2^{3\omega(q)}\ll_\ep B^\ep$
lattices, as required for the lemma. 

Any lattice on which $a_1x_1^2y_1^3+a_2x_2^2y_2^3$ vanishes
identically cannot contribute to 
$S(q,\kappa)$ and may therefore be discarded. Moreover the points on
any lattice $\mathsf{M}$ satisfy 
equations of the type (\ref{heq}). Thus if $x_1=x_2=0$ for a point on
such a lattice one sees that 
$y_1=y_2=0$, and similarly, if $y_1=y_2=0$ we will have
$x_1=x_2=0$. This concludes the proof 
of the lemma.
\end{proof}

\section{Awkward Cases}

Before proceeding further we consider points corresponding to the exceptional lines arising from
Lemma \ref{except}. Since 
\[(a_1a_2)^2\{a_1x_1^2y_1^3+a_2x_2^2y_2^3\}=(a_2x_1)^2(a_1y_1)^3+(a_1x_2)^2(a_2y_2)^3\]
it turns out that we will be interested (in Section \ref{S6}) in linear forms for which 
$L_1(u,v)=a_2x_1$, $L_2(u,v)=a_1x_2$,
$M_1(u,v)=a_1y_1$ and $M_2(u,v)=a_2y_2$.
\begin{lemma}\label{4-5}
Let $S_1(B)$ be the set of points $(x_1,x_2,x_3,y_1,y_2,y_3)\in S_0(B)$ for which there exist 
non-zero rationals $g$ and $\nu$ such that
\[a_2x_1=\nu\{4a_1y_1-5g^2a_2y_2\},\;\;\;\mbox{and}\;\;\; a_1x_2=\nu g^3\{5a_1y_1-4g^2a_2y_2\},\]
and let $S_1^{*}(B)$ be the set of points $(x_1,x_2,y_1,y_2)$ for which there is a corresponding
pair $(x_3,y_3)$ with $(x_1,x_2,x_3,y_1,y_2,y_3)\in S_1(B)$.  Then
\[ \card S_1(B)=\card S_1^*(B) \ll B^{1/2}.\]
\end{lemma}
\begin{remark*}
In fact one can easily prove a rather better bound, but this is more than enough for our purposes.
\end{remark*}
\begin{proof}
We write $g=\alpha/\beta$ with $\gcd(\alpha,\beta)=1$ and $\nu/\beta^5=\gamma/\delta$ with
$\gcd(\gamma,\delta)=1$, so that
\beql{e1}
a_2x_1\delta=\gamma\beta^3\{4a_1\beta^2 y_1-5a_2 \alpha^2 y_2\}
\eeq
and
\beql{e2}
a_1x_2\delta=\gamma\alpha^3\{5a_1\beta^2 y_1-4a_2 \alpha^2 y_2\}.
\eeq
Thus $\gamma$ divides both $a_2x_1\delta$ and $a_1x_2\delta$. Since
$\gamma$ and $\delta$ are  
coprime we then see from the constraints (\ref{gcd}) that $\gamma=\pm
1$, and indeed by changing the signs 
of $\alpha$ and $\beta$ if necessary, we may assume that $\gamma=1$.  

We next proceed to prove that $\delta|144$.  We begin by eliminating
$y_2$ from (\ref{e1}) to deduce that $\delta|9(\alpha\beta)^5 a_1y_1$. Similarly, 
eliminating $y_1$ yields $\delta|9(\alpha\beta)^5 a_2y_2$, 
and since the constraints (\ref{gcd}) show that
$\gcd(a_1y_1,a_2y_2)=1$ we conclude that  
$\delta|9(\alpha\beta)^5$.

Suppose now that $p^e||\delta$. The prime $p$ cannot divide both
$\alpha$ and $\beta$, and we may  therefore 
assume without loss of generality that $p\nmid\beta$, say. Let
$p^f||\alpha$. Then $p^{\min(e,2f)}|4a_1y_1$, by 
(\ref{e1}). However the conditions of Theorem \ref{thm2} ensure that
$a_1y_1$ is square-free, so that we must 
have 
\beql{in1}
\min(e,2f)\le \left\{\begin{array}{cc} 1, & p\ge 3,\\ 3,\ & p=2.
\end{array} \right. 
\eeq
Suppose next that $p\ge 3$. If $f=0$ then $p\nmid\alpha\beta$, and so
$p^e|9$, since $\delta|9(\alpha\beta)^5$. 
Thus if $f=0$ either $p=3$ and $e\le 2$ or $p>3$ and $e=0$. On the
other hand, if $f\ge 1$ the inequality  
(\ref{in1}) yields $e\le 1$, whence $p^2|a_1x_2$ by (\ref{e2}). Hence
$p|x_2$ since $a_1$ is square-free. 
Moreover, if $e=1$ and $f\ge 1$ the equation (\ref{e1}) shows that
$p|a_1y_1$. This however is impossible 
because $x_2$ is coprime to $a_1y_1$ in Theorem \ref{thm2}. We
therefore conclude that $e=0$ whenever  
$p\ge 3$ and $f\ge 1$. Thus if $p$ is odd, the exponent $e$ can only
be positive when $p=3$, in which  
case $e\le 2$.

Finally, for $p=2$, if $e\ge 3f+2$ the equation (\ref{e2}) would show
that $2^2|5a_1\beta^2 y_1-4a_2 \alpha^2 y_2$, 
whence $4|a_1y_1$. Since $a_1y_1$ must be square-free we conclude that
$e\le 3f+1$, and the second part 
of (\ref{in1}) then shows that $e\le 4$. We may now deduce that
$\delta|144$ as claimed. 

To complete the proof of the lemma we now observe that $\alpha^3\ll
|a_1|X_2$ and $\beta^3\ll |a_2|X_1$, by 
(\ref{e2}) and (\ref{e1}) respectively.  Moreover these equations show that 
the point $(a_1\beta^2 y_1,a_2\alpha^2 y_2)$ lies in an
ellipse, centred on the origin, of area
$O(|a_1a_2|X_1X_2|\alpha\beta|^{-3})$, and hence that $(y_1,y_2)$ is  
restricted to an ellipse, also centred on the origin, of area
$O(X_1X_2|\alpha\beta|^{-5})$. Such an ellipse 
contains 
\[\ll 1+X_1X_2|\alpha\beta|^{-5}\]
lattice points with $\gcd(y_1,y_2)=1$, by Lemma \ref{DASF}.
Given $\alpha,\beta$ and $\delta$, the values of $y_1$ and $y_2$
determine $x_1$ and $x_2$ via 
(\ref{e1}) and (\ref{e2}), and there is then at most one choice for
$x_3$ and $y_3$ satisfying the conditions of  
Theorem \ref{thm2}. Thus the total contribution to
$N_0(B)$ is
\begin{eqnarray*}
  &\ll& \sum_{|\alpha|\ll (|a_1|X_2)^{1/3}}\sum_{|\beta|\ll
    (|a_2|X_1)^{1/3}}\{1+X_1X_2|\alpha\beta|^{-5}\}\\
&\ll& (|a_1a_2|X_1X_2)^{1/3}+X_1X_2\\
&\ll& (B/X_1Y_1^3)^{1/3}(B/X_2Y_1^3)^{1/3}+X_1X_2
\end{eqnarray*}
by (\ref{range1}). The required bound then follows from (\ref{range2}).

\end{proof}

Another awkward situation is that in which one of the lattices $\mathsf{M}$ in
Lemma \ref{lemma:qkappa} contains a short vector for which
$a_1x_1^2y_1^3+a_2x_2^2y_2^3=0$. Our next result shows that this case contributes 
relatively little to $S_0(B)$.
\begin{lemma}\label{awkward}
Define
\beql{Ddef}
D=\mathrm{Diag}(X_1,X_2,Y_1,Y_2),
\eeq
and let $Q(B;R)$ be the set of pairs $(q,\kappa)$ for which there is
an integer vector $(x_1,x_2,y_1,y_2)$ 
satisfying $a_2x_1\equiv\kappa^3 a_1 x_2\modd{q}$ and
$a_1x_1^2y_1^3+a_2x_2^2y_2^3=0$, 
with $(x_1,x_2)\not=(0,0)$, $(y_1,y_2)\not=(0,0)$, and
$||D^{-1}(x_1,x_2,y_1,y_2)||_\infty\le R$.  
Then
\beql{QBR}
\card Q(B;R)\ll_\ep R(Y_1Y_2)^{-3/4}Y_3B^{1/2+\ep}
\eeq
for any fixed $\ep>0$. 

Moreover
\beql{Sqk}
\card S(q,\kappa)\ll_\ep Y_1Y_2Y_3^{-1}B^\ep
\eeq
for any pair $(q,\kappa)$ and any fixed $\ep>0$. Thus the number 
of points 
\[(x_1,x_2,x_3,y_1,y_2,q)\in S_0(B)\]
for which 
the corresponding pair $(q,\kappa)$ lies in $Q(B;R)$ is
\beql{S4b}
\ll_\ep R(Y_1Y_2)^{1/4}B^{1/2+\ep},
\eeq
for any fixed $\ep>0$.
\end{lemma}
\begin{remark*}
The reader should notice that in the definition of $Q(B;R)$ we do not impose the conditions
(\ref{dyad}) or (\ref{gcd}), or require $a_iy_i$ to be square-free.
In the Critical Case (\ref{S4b}) is $O_\ep(RB^{3/5+\ep})$, giving a non-trivial bound as soon
as $R$ is small.
\end{remark*}
\begin{proof}
Since $(a_2x_1)^2(a_1y_1)^3=-(a_1x_2)^2(a_2y_2)^3$ we deduce that
$a_2x_1=ju_1^3$ and  
$a_1x_2=ju_2^3$ for suitable integers $j,u_1$ and $u_2$, with $\gcd(u_1,u_2)=1$.
We then have $ju_1^3\equiv \kappa^3ju_2^3\modd{q}$, and on setting
$h=\gcd(j,q)$ we find that  
$u_1^3\equiv \kappa^3u_2^3\modd{q/h}$. It follows that if $u_1,u_2,h$
and $q/h$ are given, there 
will only be $O_\ep(B^\ep)$ choices for $\kappa$ modulo $q/h$, and
hence $O_\ep(hB^\ep)$ 
choices modulo $q$. Moreover there are $O(Y_3/h)$ possibilities 
for $q/h\le Y_3/h$. Thus $u_1,u_2$ and $h$ determine $O_\ep(Y_3B^\ep)$
choices for $q$ and $\kappa$. 

Since $a_1$ and $a_2$ are square-free and coprime to $q$ the relations
$a_2x_1=ju_1^3$ and  
$a_1x_2=ju_2^3$ show that $hu_1^2|x_1$ and $hu_2^2|x_2$, showing that
$hu_1^2\le X_1R$ and 
$hu_2^2\le X_2R$. Thus, given $h$, there are $O(Rh^{-1}(X_1X_2)^{1/2})$
choices for $u_1$ and 
$u_2$. Then, summing over $h$, we see that there are
$O(R(X_1X_2)^{1/2}\log B)$ available 
 triples $u_1,u_2,h$, producing a total
 $O_\ep(R(X_1X_2)^{1/2}Y_3B^{2\ep})$ choices for $q$ and 
 $\kappa$. The estimate (\ref{QBR}) then follows from (\ref{range3}) on re-defining $\ep$.

We turn now to the proof of (\ref{Sqk}).
By Lemma \ref{DASF}, the second of the congruences (\ref{c1}) has
$O(Y_1Y_2/q+1)$ relevant  
solution pairs $(y_1,y_2)$, and since $\tfrac12 Y_3<q=y_3\le Y_3$ we see from
(\ref{range2}) that this is $O(Y_1Y_2/Y_3)$. For each of these pairs $(y_1,y_2)$ the 
values of $x_1$ and $x_2$ will lie on a lattice 
$x_2\equiv tx_1\modd{q^3}$, where $a_1y_1^3+a_2t^2y_2^3\equiv
0\modd{q^3}$. This latter 
congruence has $O_\ep(B^\ep)$ solutions $t$ modulo $q^3$, while the
lattice $x_2\equiv tx_1\modd{q^3}$ 
has $O(X_1X_2q^{-3}+1)$ relevant points $(x_1,x_2)$ via
Lemma \ref{DASF} again. A further application of (\ref{range2})
shows that this is $O(1)$, and it follows that 
$\card S(q,\kappa)\ll_\ep Y_1Y_2Y_3^{-1}B^\ep$ as required.

For the final claim of the lemma one merely has to combine (\ref{QBR})
and (\ref{Sqk}) and re-define $\ep$.
\end{proof}

\section{Basis Vectors Associated to $\mathsf{M}$}\label{S6}

Now that we have handled the exceptional cases covered by Lemmas \ref{4-5} and \ref{awkward},
we are ready to count points $(x_1,x_2,y_1,y_2)\in S(q,\kappa)$ which lie in
a particular lattice $\mathsf{M}$ in Lemma \ref{lemma:qkappa}. This will
require us to investigate $\det(\mathsf{M})$, and in particular to understand
how often it might be small.

In fact it is convenient to work with
$D^{-1}\mathsf{M}$ rather than $\mathsf{M}$ itself, where
$D$ is given by (\ref{Ddef}).
We choose a basis $\mathbf{g}$, $\mathbf{h}$ for $\mathsf{M}$ so that the
corresponding basis
$D^{-1}\mathbf{g},D^{-1}\mathbf{h}$ for $D^{-1}\mathsf{M}$ is minimal, in the sense
that  $||D^{-1}\mathbf{g}||_\infty$ and $||D^{-1}\mathbf{h}||_\infty$ are the successive
minima $\lambda_1$ and $\lambda_2$ of $D^{-1}\mathsf{M}$, with
respect to the $||\cdot||_\infty$ norm. We then have
\beql{rs0}
||D^{-1}\mathbf{g}||_\infty\le ||D^{-1}\mathbf{h}||_\infty
\eeq
and
\beql{rs}
\det\left(D^{-1}\mathsf{M}\right)\ll
||D^{-1}\mathbf{g}||_\infty||D^{-1}\mathbf{h}||_\infty\ll \det\left(D^{-1}\mathsf{M}\right).
\eeq
Moreover these basis vectors have the property that if
$||uD^{-1}\mathbf{g}+vD^{-1}\mathbf{h}||_\infty\le 1$ then 
\[|u|\ll 1/||D^{-1}\mathbf{g}||_\infty\;\;\;\mbox{and}\;\;\;
|v|\ll 1/||D^{-1}\mathbf{h}||_\infty.\]
It turns out that the basis vectors $D^{-1}\mathbf{g}$ and $D^{-1}\mathbf{h}$ will
typically have length $O(1)$, so the reader may wish to focus their attention on
this situation.
\begin{lemma}\label{useSS}
If $\mathsf{M}$ has a basis $\mathbf{g},\mathbf{h}$ as above, then 
\beql{uSS}
\card\big(\left(S(q,\kappa)\cap \mathsf{M}\right)\setminus S_1^*(B)\big)\ll_\ep  
 B^\ep\left\{||D^{-1}\mathbf{g}||_\infty^{-2/3}||D^{-1}\mathbf{h}||_\infty^{-2/3}+1\right\}.
 \eeq
\end{lemma}
Here $S_1^*(B)$ is the subset of $S(q,\kappa)$ defined in Lemma \ref{4-5}.
\begin{proof}
We write the quadruple $(x_1,x_2,y_1,y_2)\in S(q,\kappa)\cap \mathsf{M}$ as
$u\mathbf{g}+v\mathbf{h}$.  Thus we obtain
integral linear forms such that
\[x_1=x_1(u,v),\; x_2=x_2(u,v),\; y_1=y_1(u,v),\;\mbox{and}\; y_2=y_2(u,v).\]
Moreover the condition (\ref{dyad}) shows that
$||D^{-1}(x_1,x_2,y_1,y_2)||_\infty\le 1$, 
so that $||uD^{-1}\mathbf{g}+vD^{-1}\mathbf{h}||_\infty\le 1$. We therefore have
$|u|\le U$ and $|v|\le V$ with $U\ll 1/||D^{-1}\mathbf{g}||_\infty$ and
$V\ll 1/||D^{-1}\mathbf{h}||_\infty$. 

Since $x_1$ and $x_2$ must be coprime we will have $\gcd(u,v)=1$, Thus if 
$V<1$ for example, we see that $v=0$ and $u=\pm1$, giving us $O(1)$ points in total.
We therefore suppose that $U,V\ge 1$ in what follows. Moreover, if $x_1(u,v)$ 
and $x_2(u,v)$ are proportional one sees that they can give at most one pair of 
positive coprime values of $x_1$ and $x_2$, and then Lemma \ref{XXY} shows that 
there are $O_\ep(B^\ep)$ corresponding values of $x_3,y_1$ and $y_2$, in light 
of (\ref{range2}). Likewise, if
$y_1(u,v)$  and $y_2(u,v)$ are proportional we get an overall contribution
$O_\ep(B^\ep)$ in Lemma \ref{useSS}, by employing Lemma \ref{YYY}.

We now write
\[L_1(u,v)=a_2x_1(u,v),\; L_2(u,v)=a_1x_2(u,v),\]
\[M_1(u,v)=-a_1a_3qy_1(u,v),\;\mbox{and} \; M_2(u,v)=-a_2a_3qy_2(u,v).\]
Then if we define $F(u,v)=L_1(u,v)^2M_1(u,v)^3+L_2(u,v)^2M_2(u,v)^3$ we will have
\begin{eqnarray*}
F(u,v)&=&-a_3q^3(a_1a_2a_3)^2\{a_1x_1(u,v)^2y_1(u,v)^3+a_2x_2(u,v)^2y_2(u,v)^3\}\\
&=&(a_1a_2a_3q^3x_3)^2.
\end{eqnarray*}

We are now ready to apply Lemma \ref{except}. We wish to count the
number of $(u,v)$ 
for which $F(u,v)$ is a non-zero square.  We first claim that the
coefficients of  
the forms $L_1,L_2,M_1$ and $M_2$ are bounded by powers of $B$. Since
$\mathbf{g}$  
is a non-zero vector in $\mathsf{M}\subseteq\Z^4$ we have
$||\mathbf{g}||_\infty\ge 1$. 
Thus $||D^{-1}\mathbf{g}||_\infty$ is bounded below by a negative power of
$B$, and then (\ref{rs0}) and 
(\ref{rs}) show that $||D^{-1}\mathbf{g}||_\infty$ and
$||D^{-1}\mathbf{h}||_\infty$ are bounded above  
by a power of $B$. It then follows that the coefficients of $x_1(u,v),
x_2(u,v), y_1(u,v)$ 
and $y_2(u,v)$ are bounded by powers of $B$, and hence so too are the
coefficients of $L_1,L_2,M_1$ and $M_2$, and therefore also of $F$.

In order to apply Lemma \ref{except} we need to know that neither of
the linear forms  
$L_1$ or $M_1$ is proportional to  either $L_2$ or $M_2$. We have already dealt with
$x_1(u,v)$ and $x_2(u,v)$, and with $y_1(u,v)$ and $y_2(u,v)$, so we look at the 
case in which $L_1$ and $M_2$ are proportional, both integer multiples of an integral
linear form $\ell(u,v)$, say.  Since $\gcd(x_1,y_2)=1$ we must have
$\ell(u,v)=\pm 1$. We may parameterize the solutions of this linear equation
as integral linear functions $u=u_1 t+u_2$, $v=v_1 t+v_2$, where $t$ is
an integer variable of size $|t|\ll U+V$. Then $F(u,v)$ takes the shape
$A_1X_1(t)^2+A_2Y_2(t)^3$ with non-zero integer polynomials $X_1(t),Y_2(t)$ of degree at
most 1.  According to
Lemma \ref{quintic2} the number of values for $t$ which produce a non-zero square is
\beql{es1}
\ll (U+V)^{1/2}\log B\ll (UV)^{2/3}\log B
\eeq
unless $A_1X_1(t)^2+A_2Y_2(t)^3$ is a constant multiple of a square. However this cannot 
happen unless $X_1(t)$ and $Y_2(t)$ are both constant. It follows that we have at most
1 solution when $L_1$ and $M_2$ are proportional. Naturally, when $L_2$ and $M_1$
are proportional a similar argument applies.

In the statement of Lemma \ref{useSS} we exclude points in $S_1^*(B)$, these
corresponding to points on the various exceptional 
lines described in Lemmas \ref{except} and \ref{4-5}. Thus Lemma
\ref{except} gives us the bound 
\[\ll (UV)^{2/3}(\log B)^2+1\]
for the number of remaining points, whether or not two of the forms
$L_1,L_2,M_1$ and $M_2$ are 
proportional. To complete the proof of the lemma we merely note that
\[UV\ll ||D^{-1}\mathbf{g}||_\infty^{-1}||D^{-1}\mathbf{h}||_\infty^{-1}.\]
\end{proof}

We now need to estimate how frequently small generators
$\mathbf{g},\mathbf{h}$ arise. 
As a preliminary observation, we note that, if
$\mathbf{g}=(x_1,x_2,y_1,y_2)\in\mathsf{M}$,  
then we cannot have
$x_1=x_2=0$ or $y_1=y_2=0$, by the final part of Lemma \ref{lemma:qkappa}. Thus
Lemma \ref{awkward} handles all cases in which $a_1x_1^2y_1^3+a_2x_2^2y_2^3=0$. 
Our next result handles the alternative case, for which $a_1x_1^2y_1^3+a_2x_2^2y_2^3\not=0$.
\begin{lemma}\label{gensize}
Write $\mathbf{g}=(x_1,x_2,y_1,y_2)$. Then the number of triples
$\left(\mathbf{g},q,\kappa\right)$  
with $a_1x_1^2y_1^3+a_2x_2^2y_2^3\not=0$, for which 
$\mathbf{g}$ satisfies (\ref{c1}), (\ref{c2}), (\ref{c3}) and
(\ref{c4}), and such that 
\beql{DR}
||D^{-1}\mathbf{g}||_\infty\le R, 
\eeq
is $O_\ep\left((Y_1Y_2)^{-3/2}Y_3^{-3}R^7B^{2+\ep}\right)$, for any fixed $\ep>0$. 

In particular, the number of pairs $(q,\kappa)$ for which such a vector $\mathbf{g}$
exists is also $O_\ep\left((Y_1Y_2)^{-3/2}Y_3^{-3}R^7B^{2+\ep}\right)$, for any fixed $\ep>0$. 
\end{lemma}
\begin{remark*}\label{Rgs}
Although the lemma gives a bound for the number of 
relevant triples $(\mathbf{g},q,\kappa)$, it does not estimate the number of corresponding
lattices $\mathsf{M}$ that might contain a suitable $\mathbf{g}$. This is because a given  
$\mathbf{g}$ might lie in more than one lattice $\mathsf{M}$.
\end{remark*}
\begin{proof}
We begin by classifying the triples $\left(\mathbf{g},q,\kappa\right)$ 
according to the value of $r=\gcd(y_1,y_2,q)$.  
Lemma \ref{TheLattice} then shows that
\[18a_1x_1^2y_1^3+18a_2x_2^2y_2^3=fq^3r^2\]
for some integer $f$, which will be non-zero. We see from (\ref{dyad}) and (\ref{range1}) that
the left hand side 
has absolute value $O(R^5B)$ and since $q=y_3>\tfrac12 Y_3$ we find that 
\[f\ll Y_3^{-3}R^5Br^{-2}.\]
The triple $(y_1/r,y_2/r,q/r)$ is a coprime solution to the equation 
(\ref{ceq}) with coefficients 
\[a=18a_1x_1^2,\;\;\; b=18a_2x_2^2,\;\;\;\mbox{and}\;\;\; c=-fr^2.\]
Moreover $a,b$ and $c$ are non-zero by our assumptions. Thus Theorem
\ref{cubic}  
shows that  $x_1,x_2,r$ and $f$ determine $O_\ep(B^\ep)$ possibilities
for $y_1,y_2$ and $q$.  
There are 
\[\ll X_1X_2Y_3^{-3}R^7Br^{-2}\ll (Y_1Y_2)^{-3/2}Y_3^{-3}R^7B^2r^{-2}\]
possible choices for $x_1,x_2$ and $f$, by (\ref{range3}), so that the total 
number of possibilities for $\mathbf{g}$ and $q$ is
$O_\ep((Y_1Y_2)^{-3/2}Y_3^{-3}R^7B^{2+\ep}r^{-2})$, 
for each $r$.

The second of the congruences 
(\ref{c1}) shows that $qr^{-1}|\kappa^2 a_1y_1r^{-1}+a_2y_2r^{-1}$,
which determines $O_\ep(B^\ep)$ values for $\kappa$ modulo $qr^{-1}$, and hence
$O_\ep(rB^\ep)$ values modulo $q$. We therefore have 
$O_\ep((Y_1Y_2)^{-3/2}Y_3^{-3}R^7B^{2+2\ep}r^{-1})$
choices for $\left(\mathbf{g},q,\kappa\right)$, for each value of
$r$. The lemma then  
follows on summing for $r\le Y_3$ and re-defining $\ep$.
\end{proof}

As noted above, a generator $\mathbf{g}=(x_1,x_2,y_1,y_2)$ cannot have
$x_1=x_2=0$ or $y_1=y_2=0$, by the final part of Lemma \ref{lemma:qkappa}.
If $a_1x_1^2y_1^3+a_2x_2^2y_2^3=0$ the pair $(q,\kappa)$ will therefore be in the set
$Q(B;R)$ defined in Lemma \ref{awkward}. Thus Lemma \ref{gensize} has the 
following corollary.

\begin{lemma}\label{Rest}
For each pair $q,\kappa$ let $\nu_1(q,\kappa;R)$ be the number of 
non-zero vectors
$\mathbf{g}$ with $||D^{-1}\mathbf{g}||_\infty\le R$ that lie in some
lattice $\mathsf{M}$ in $\mathcal{L}(q,\kappa)$. Then
\[\sum_{(q,\kappa)\not\in Q(B;R)}\nu_1(q,\kappa;R)\ll_\ep
(Y_1Y_2)^{-3/2}Y_3^{-3}R^7B^{2+\ep},\]
for any fixed $\ep>0$. 
\end{lemma}

For the second generator $\mathbf{h}$ we do not need Lemma \ref{awkward}.
\begin{lemma}\label{Sest}
For each pair $q,\kappa$ let $\nu_2(q,\kappa;S)$ be the number of 
non-zero vectors $\mathbf{h}$ with $||D^{-1}\mathbf{h}||_\infty\le S$ that lie in some
lattice $\mathsf{M}$ in $\mathcal{L}(q,\kappa)$, and with the further property
that there is a second vector $\mathbf{g}\in\mathsf{M}$, not proportional to
$\mathbf{h}$, and such that $||D^{-1}\mathbf{g}||_\infty\le ||D^{-1}\mathbf{h}||_\infty$.

Then 
\[\sum_{(q,\kappa)}\nu_2(q,\kappa;S)\ll_\ep (Y_1Y_2)^{-3/2}Y_3^{-3}S^7B^{2+\ep},\]
for any fixed $\ep>0$.
\end{lemma}
\begin{proof}
For pairs $(q,\kappa)$ for which $\mathbf{h}=(x_1,x_2,y_1,y_2)$ with 
\[a_1x_1^2y_1^3+a_2x_2^2y_2^3\not=0\]
we may apply Lemma \ref{gensize} directly. 
In general we consider vectors 
\[\mathbf{h}'=n\mathbf{g}+\mathbf{h}=(x_1(n),x_2(n),y_1(n),y_2(n)).\]
We then set $f(n)=a_1x_1(n)^2y_1(n)^3+a_2x_2(n)^2y_2(n)^3$. The
polynomial $f(X)$ cannot  
vanish identically, by the properties of $\mathsf{M}$ established in
Lemma \ref{lemma:qkappa}. 
Since $f$ has degree 5 there is therefore some positive integer $n\le
6$ such that $f(n)\not=0$. 
This produces a vector $\mathbf{h}'\in\mathbf{M}$, with
\[||D^{-1}\mathbf{g}||_\infty\le 7||D^{-1}\mathbf{h}||_\infty\le 7S\]
to which Lemma \ref{gensize} applies. The result then follows
on using this as a replacement for $\mathbf{h}$.
\end{proof}

We can now consider how many lattices $\mathsf{M}$ have basis vectors
of a given size. We will 
use the following result to control the first term on the right of (\ref{uSS}).
\begin{lemma}\label{sizeRS}
The number of lattices
\[\mathsf{M}\in\bigcup_{(q,\kappa)\not\in Q(B;R)}\mathcal{L}(q,\kappa)\]
which have basis vectors $\mathbf{g},\mathbf{h}$ with
$||D^{-1}\mathbf{g}||_\infty\le ||D^{-1}\mathbf{h}||_\infty$ and 
\[ ||D^{-1}\mathbf{g}||_\infty\le R,\;\;\; ||D^{-1}\mathbf{h}||_\infty\le S,\]
is 
\[\ll_\ep (Y_1Y_2)^{-5/3}Y_3^{-4}(RS)^{7/2}B^{7/3+\ep}\]
for any fixed $\ep>0$.
\end{lemma}
\begin{remark*}
This bound will handle the first term on the right of (\ref{uSS}), when 
$\mathbf{g}$ and $\mathbf{h}$ are not too big.
\end{remark*}
\begin{proof}
Although it seems possible that a given vector $\mathbf{g}$ may lie
in more than one lattice $\mathsf{M}$, 
it is certainly the case that $\mathbf{g}$ and $\mathbf{h}$ together
do determine $\mathsf{M}$. Thus the number 
of lattices in $\mathcal{L}(q,\kappa)$ having basis vectors
$\mathbf{g},\mathbf{h}$ as in the lemma is at most 
$\nu_1(q,\kappa;R)\nu_2(q,\kappa;S)$.  Alternatively, Lemma
\ref{lemma:qkappa} shows that we may bound  
this number by $\card \mathcal{L}(q,\kappa)\le L$, with 
\[L\ll_\ep (Y_1Y_2)^{-1/3}Y_3^{-2}B^{2/3+\ep}.\]
Thus the number of lattices satisfying the conditions of the lemma is 
\begin{eqnarray*}
  &\le& \sum_{(q,\kappa)\not\in Q(B;R)}
  \min\big(\nu_1(q,\kappa;R)\nu_2(q,\kappa;S)\, ,\,L\big)\\
  &\le& \sum_{(q,\kappa)\not\in Q(B;R)}\big(\nu_1(q,\kappa;R)
  \nu_2(q,\kappa;S))^{1/2}L^{1/2}\\
&\le& L^{1/2}\left(\sum_{(q,\kappa)\not\in Q(B;R)}\nu_1(q,\kappa;R)\right)^{1/2}
\left(\sum_{(q,\kappa)\not\in Q(B;R)}\nu_2(q,\kappa;S)\right)^{1/2},
\end{eqnarray*}
by Cauchy's inequality. The result then follows from Lemmas \ref{Rest}
and \ref{Sest} on re-defining $\ep$. 
\end{proof}

\section{Combining all our Bounds}
In this section we combine the various bounds we have obtained.
We remind the reader that, given the coefficients $a_i$ in
(\ref{axy}), the values of $x_1,x_2,y_1,y_2$ determine at most one choice for $x_3$ 
and $y_3$, whence it suffices to count points $(x_1,x_2,y_1,y_2)$. We will write
$S_0(B)$ as a union
\[S_0(B)=S_1(B)\cup S_2(B)\cup S_3(B)\cup S_4(B)\cup S_5(B)\]
and estimate $\card S_i(B)$ separately for each index $i$. We already have
\beql{S1B}
 \card S_1(B) \ll B^{1/2}
 \eeq
from Lemma \ref{4-5},
accounting for points lying on the ``bad'' lines in Lemma \ref{except}. 
We define $S_2(B)$ to be the subset of $S_0(B)$ for which $(x_1,x_2,y_1,y_2)$ lies 
in the set $S_0(q,\kappa)$ in Lemma \ref{lemma:qkappa}. Since there are $O(Y_3^2)$
possibilities for $(q,\kappa)$ we see from (\ref{S0qk}) that
\beql{S2B}
\card S_2(B)\ll_\ep (Y_1Y_2)^{-1/3}B^{2/3+\ep},
\eeq
for any fixed $\ep>0$.

The remaining points of $S_0(B)$ are covered by Lemma \ref{useSS}. 
We give ourselves a parameter $K\in (0,1]$ to be chosen later, and take $S_3(B)$
to be the set of points in $S_0(B)\setminus\left(S_1(B)\cup S_2(B)\right)$ 
belonging to lattices $\mathsf{M}$ for which the basis vectors satisfy
$||D^{-1}\mathbf{g}||_\infty ||D^{-1}\mathbf{h}||_\infty\ge K$.
Since  Lemma \ref{lemma:qkappa}
shows that the total number of lattices $\mathsf{M}$ is
\[\ll_\ep  (Y_1Y_2)^{-1/3}B^{2/3+\ep},\]
we see that (\ref{uSS}) yields
\beql{S3B}
\card S_3(B)\ll_\ep K^{-2/3} (Y_1Y_2)^{-1/3}B^{2/3+2\ep}.
\eeq

It remains to handle lattices $\mathsf{M}$ for which 
$||D^{-1}\mathbf{g}||_\infty ||D^{-1}\mathbf{h}||_\infty\le K$.
For these we do a dyadic subdivision for the ranges of 
$||D^{-1}\mathbf{g}||_\infty$ and  $||D^{-1}\mathbf{h}||_\infty$, writing
\beql{RSd}
\tfrac12 R<||D^{-1}\mathbf{g}||_\infty \le R\;\;\;\mbox{and}
\;\;\; \tfrac12 S< ||D^{-1}\mathbf{h}||_\infty\le S,
\eeq
with $RS\ll K$. Since $||\mathbf{g}||_\infty\le||\mathbf{h}||_\infty$ we will have
$R\ll K^{1/2}$. Thus if we define $R_0=K^{1/2}\log B$ we will have $R\le R_0$ as long 
as $B$ is suitably large. We now define $S_4(B)$ to be the set of points
$(x_1,x_2,x_3,y_1,y_2,q)\in S_0(B)$ for which 
the corresponding pair $(q,\kappa)$ lies in $Q(B;R_0)$, as described in 
Lemma \ref{awkward}. It then follows from (\ref{S4b}) that
\beql{S4B}
\card S_4(B)\ll_\ep K^{1/2}(Y_1Y_2)^{1/4}B^{1/2+2\ep},
\eeq
for any fixed $\ep>0$.

Our final set $S_5(B)$ consists of points in $S_0(B)$ belonging to lattices 
$\mathsf{M}$ whose generators lie in the ranges (\ref{RSd}) with $RS\ll K$, and for which 
the corresponding values $(q,\kappa)$ do not lie in $Q(B;R_0)$. For these we have
$(q,\kappa)\not\in Q(B;R)$ {\em a fortiori}, so that the number of relevant lattices
$\mathsf{M}$ with generators in the ranges (\ref{RSd}) is
\[\ll_\ep (Y_1Y_2)^{-5/3}Y_3^{-4}(RS)^{7/2}B^{7/3+\ep}\]
by Lemma \ref{sizeRS}. Moreover each such lattice provides 
$O_\ep\left((RS)^{-2/3}B^\ep\right)$ points, by Lemma \ref{useSS}, whence 
\[\card S_5(B)\ll_\ep (Y_1Y_2)^{-5/3}Y_3^{-4}B^{7/3+2\ep}
\sum_{R,S}(RS)^{7/2-2/3},\]
the sum being over dyadic values of $R$ and $S$ with $RS\ll K$.
We therefore conclude that
\beql{S5B}
\card S_5(B)\ll_\ep (Y_1Y_2)^{-5/3}Y_3^{-4}B^{7/3+2\ep}K^{17/6}.
\eeq

We may now combine (\ref{S1B}), (\ref{S2B}), (\ref{S3B}), (\ref{S4B}), and (\ref{S5B}), and
replace $\ep$ by $\ep/2$, to show that if $0<K\le 1$ then
\[N_0(B)=\card S_0(B)\ll_\ep\left\{K^{-2/3} (Y_1Y_2)^{-1/3}B^{2/3}+
K^{1/2}(Y_1Y_2)^{1/4}B^{1/2}\right.\]
\[\hspace{4cm}\left. \mbox{}+(Y_1Y_2)^{-5/3}Y_3^{-4}B^{7/3}K^{17/6}\right\}B^{\ep}\]
for any fixed $\ep>0$, and any parameters in the ranges (\ref{range2}). We choose
\[K=(Y_1Y_2)^{8/21}Y_3^{8/7}B^{-10/21}\]
so as to equate the first and third terms on the right, and note that $K\le 1$ by (\ref{range2}).
This choice gives us
\[N_0(B)\ll_\ep\left\{(Y_1Y_2)^{-37/63}Y_3^{-16/21}B^{62/63}+
(Y_1Y_2)^{37/84}Y_3^{4/7}B^{11/42}\right\}B^{\ep}.\]
The second term on the left is smaller than the first, by (\ref{range2}).
We may therefore conclude that
\[N_0(B)\ll_\ep (Y_1Y_2)^{-37/63}Y_3^{-16/21}B^{62/63+\ep}.\]

Although we said at the outset that we would work with congruences to modulus $y_3^3$
we could just as well have used $y_1^3$ or $y_2^3$, giving us analogous bounds with 
$Y_1,Y_2$ and $Y_3$ permuted. Thus
\begin{eqnarray*}
N_0(B)^3&\ll_\ep& (Y_2Y_3)^{-37/63}Y_1^{-16/21}B^{62/63+\ep}\times
(Y_1Y_3)^{-37/63}Y_2^{-16/21}B^{62/63+\ep}\\
&&\hspace{2cm}\mbox{}\times(Y_1Y_2)^{-37/63}Y_3^{-16/21}B^{62/63+\ep}\\
&=& (Y_1Y_2Y_3)^{-122/63}B^{62/21+3\ep},
\end{eqnarray*}
so that $N_0(B)\ll_\ep (Y_1Y_2Y_3)^{-122/189}B^{62/63+\ep}$.
We now see from Lemma \ref{YYY} that either $N_0(B)\ll_\ep B^{1/2+\ep}$ or
\begin{eqnarray*}
N_0(B) &=& N_0(B)^{122/311}N_0(B)^{189/311}\\
&\ll_\ep &(Y_1Y_2Y_3B^\ep)^{122/311}
\left\{(Y_1Y_2Y_3)^{-122/189}B^{62/63+\ep}\right\}^{189/311}\\
&=&B^{186/311+\ep}.
\end{eqnarray*}
In view of (\ref{0}) this suffices for Theorem \ref{thm2}, since $186/311=3/5-3/1555$.
\bigskip

We conclude by observing that the critical case for any improvement on Theorem \ref{thm2}
is now that in which $|a_1|=|a_2|=|a_3|=1$,
\[X_1=X_2=X_3=B^{125/622}=B^{1/5+3/3110},\]
and
\[Y_1=Y_2=Y_3=B^{62/311}=B^{1/5-1/1555}.\]

Mathematical Institute, 

Radcliffe Observatory Quarter, 

Woodstock Road, 

Oxford  OX2~6GG,

UK
\bigskip

 {\tt rhb@maths.ox.ac.uk}

\end{document}